\documentclass[12pt]{article}%
\usepackage{amsmath, amsfonts, amsthm, color,latexsym}
\usepackage{amsmath}
\usepackage{amsfonts}
\usepackage{amssymb}
\usepackage{color} 
\allowdisplaybreaks[4]
\usepackage{graphicx}%
\providecommand{\U}[1]{\protect\rule{.1in}{.1in}}
\newtheorem{teo}{Theorem}[section]
\newtheorem{prop}[teo]{Proposition}
\newtheorem{cor}[teo]{Corollary}
\newtheorem{ex}[teo]{Example}
\newtheorem{obs}[teo]{Remark}

\newtheorem{lema}[teo]{Lemma}
\newtheorem{final remark}[teo]{Final Remark}
\newtheorem{definition}[teo]{Definition}
\newcommand{\an}{\left \Vert} 

\newcommand{\fn}{\right \Vert} 

\newcommand{\ap}{\left (} 

\newcommand{\fp}{\right )} 

\newcommand*{\longhookrightarrow}{\ensuremath{\lhook\joinrel\relbar\joinrel\rightarrow}}

\begin{document}

\title{\sc Polynomial ideals from a nonlinear viewpoint}
\date{}
\author{Geraldo Botelho\thanks{Supported by CNPq Grant
305958/2014-3 and Fapemig Grant PPM-00490-15.}  ~and Ewerton R. Torres\thanks{Supported by a CAPES postdoctoral scholarship.\hfill\newline2010 Mathematics Subject
Classification: 47L22, 46G25, 47L20, 47B10, 47A80, 46B28. \newline Keywords: Banach spaces, homogeneous polynomials, hyper-ideals, two-sided ideals.}}\maketitle

\begin{abstract} Classes of homogeneous polynomials between Banach spaces have been studied in the last three decades from the perspective of the so-called ideal property: if a polynomial $P$ belongs to a class $\cal Q$, then the composition $u \circ P \circ v$ of $P$ with {\it linear operators} $u$ and $v$ belongs to $\cal Q$ as well. In an attempt to explore the nonlinearity of the subject in a more consistent way, and taking into account recent results in the field, in this paper we propose the study of classes of homogeneous polynomials that are stable under the composition with homogeneous polynomials. Some important classes justify the study of the intermediate concept of classes of polynomials $\cal Q$ such that if $P$ belongs to $\cal Q$, $u$ is a linear operator and $Q$ is a homogeneous polynomial, then $u \circ P \circ Q$ belongs to $\cal Q$.
\end{abstract}

\section{Introduction and background}

In 1983, A. Pietsch \cite{pietsch} started the study of ideals of multilinear operators between Banach spaces as an extension of the successful theory of ideals of linear operators (operator ideals). Naturally enough, ideals of homogeneous polynomials between Banach spaces begun to be studied soon afterwards. It is difficult to point out the first occurrence of the concept of polynomial ideals, but the idea goes back at least to 1984 with the thesis \cite{teseandreas} of H.-A. Braunss. In Floret \cite[1.9]{klausathens} one can find a digression on the appearance of the concept of polynomial ideals.

For the last three decades, a huge amount of work has been done on polynomial ideals. The theory has been developed {\it pari passu} with the theory of ideals of multilinear operators (multi-ideals), and connections have been established with other topics, such as (just a few references are given): infinite dimensional holomorphy \cite{sonia, carandocollect, gonzagut},  topological tensor products \cite{carandoqjm, erhanpilar, galicerJMAA}, ultrapower stability \cite{klausathens, klauspams}, quantum information theory \cite{montanaro, davidnacho}, Dirichlet series \cite{DefantManolo, DefantPopa}, integral formulas/stable measures \cite{carandolaa}, coherence/compatibility \cite{carandocomp, joilson}, Bishop-Phelps-Bollob\'as-Lindenstrauss circle of ideas \cite{acosta, carandoJMAA},  unconditional bases \cite{carandoqjm, DefantDavid}, interpolation theory \cite{variosJFA, BJ, defantpablo}, classical inequalities (Hardy--Littlewood, Bohnenblust--Hille, Blei) \cite{danielisrael, variosJFA, danielgustavoarchiv, defantpablo, popa}, summability properties \cite{achour, mariodaniel, danieljoedsonmz}, approximation properties \cite{sonia, erhan, erhanpilar}, extension of multilinear operators and polynomials \cite{castillo, galicerJMAA, david}, Aron-Berner stability \cite{pablo, galicerJMAA}, hypercyclic convolution operators \cite{fourier, carandohyper}, homological methods \cite{castillo}, lineability/spaceability \cite{mario, danieljuan}.

Roughly speaking, a polynomial ideal is a class $\cal Q$ of homogeneous polynomials between Banach spaces such that if a polynomial $P$ belongs to $\cal Q$ and $u$ and $v$ are {\it linear operators} with compatible domains, then the composition $u \circ P \circ v$ belongs to $\cal Q$ as well. On the one hand, compositions with linear operators have the advantage of keeping the degree of homogeneity, that is, if $P$ is an $n$-homogeneous polynomial, then the composition $u \circ P \circ v$ is an $n$-homogeneous polynomial as well. On the other hand, besides of being a nonlinear theory with an excessive linear flavor, this approach disregards the fact that compositions of homogeneous polynomial with homogeneous polynomials are homogeneous polynomials as well. We are thus compelled to the consideration of classes of homogeneous polynomials $\cal Q$ such that if $P$ belongs to $\cal Q$ and $R$ and $Q$ are homogeneous polynomials with compatible domains, then the composition $R \circ P \circ Q$ belongs to $\cal Q$ as well. This stronger condition has been recently studied for some specific classes, see e.g. \cite{DefantPopa, popapalermo, popacentral}. The aim of this paper is to start the systematic study of classes of homogeneous polynomials that enjoy this stronger condition, classes which we call {\it polynomial two-sided ideals}. We believe that the study of two-sided ideals explores the nonlinearity of the subject in a more consistent, deep and fruitful way.

Like any other mathematical concept, polynomial two-sided ideals are worth studying only if many interesting and useful examples are known. In this paper we pay a special attention to the production of illustrative examples, including well studied classes as well as new classes we introduce here. It is worth mentioning that, as expected, some well studied polynomial ideals fail to be two-sided ideals; but, somewhat surprisingly, every time we meet an important polynomial ideal ${\cal Q}_1$ that is not a two-sided ideal, we immediately provide a two-sided ideal ${\cal Q}_2$ that plays, in the two-sided ideal setting, a role similar to the role played by ${\cal Q}_1$ in the polynomial ideal setting.

Scrutinizing the already studied important polynomial ideals, we sometimes found ourselves in the following situation: a given polynomial ideal $\cal Q$ fails to be a two-sided ideal but satisfies a condition that is intermediate between the conditions satisfied by polynomial ideals and two-sided ideals, namely: if a homogeneous polynomial $P$ belongs to $\cal Q$ and the linear operator $u$ and the homogeneous polynomial $Q$ have compatible domains, then the composition $u \circ P \circ Q$ also belongs to $\cal Q$. For instance, the class ${\cal P}_{\cal W}$ of weakly compact polynomials is a polynomial ideal, fails to be a two-sided ideal and satisfies this intermediate condition. In order not to leave such important classes behind, we study the classes enjoying this intermediate condition, which shall be called {\it polynomial hyper-ideals}. Observe that this kind of composition also makes sense for multilinear operators. The corresponding concept for multilinear operators, called {\it multilinear hyper-ideals}, was studied in \cite{ewerton, ewerton2}, a study from which we shall take advantage repeatedly.

The paper is organized as follows. In Section \ref{sec2} we define polynomial two-sided ideals and hyper-ideals and give some of their general properties. In Section \ref{sec:mfatat} we give plenty of examples, some of them with distinguished properties, and counterexamples. For instance, we identify the smallest polynomial two-sided ideal (and hyper-ideal), namely, the class of hyper-nuclear polynomials (cf. Theorem \ref{snmb}). The aim of Section \ref{multil} is to establish the relationship between polynomial two-sided and hyper ideals with multilinear hyper-ideals. We show that the class of polynomials generated by a multilinear hyper-ideal is a polynomial hyper-ideal and give sufficient conditions for it to be a two-sided ideal. Sections \ref{ineqmeth} and \ref{boundmeth} provide methods to generate polynomial hyper/two-sided ideals. In particular, some well known classes will arise as further examples of polynomial hyper/two-sided ideals. In Section \ref{compideals} we make a thorough study of the classical composition polynomial ideals ${\cal I} \circ {\cal P}$, where $\cal I$ is an operator ideal. We prove that ${\cal I} \circ {\cal P}$ is always a polynomial hyper-ideal and we give necessary and sufficient conditions on $\cal I$ for ${\cal I} \circ {\cal P}$ to be a polynomial two-sided ideal. Examples of operator ideals $\cal I$ satisfying such necessary and sufficient conditions are provided.

All Banach spaces are over the scalar field $\mathbb{K} = \mathbb{R}$ or $\mathbb{C}$. By ${\cal P}(^nE;F)$ we mean the Banach space of all continuous $n$-homogeneous polynomials from $E$ to $F$ endowed with the usual sup norm, which shall be denoted by $\|\cdot\|$. When $F = \mathbb{K}$ we write ${\cal P}(^nE)$, when $n = 1$ we write ${\cal L}(E;F)$ and when $n = 1$ and $F = \mathbb{K}$ we write $E'$. For background on spaces of polynomials we refer to \cite{dineen, mujica}.

By a \textit{class of polynomials} we mean a subclass $\mathcal{Q}$ of the class  $\cal P$ of all continuous homogenous polynomials between Banach spaces endowed with a function $\|\cdot\|_\mathcal{Q}\colon {\cal Q}\longrightarrow[0,\infty)$. Given classes of polynomials $\mathcal{Q}$ and $\mathcal{R}$ and a sequence of positive numbers $(C_n)_{n=1}^\infty$, the symbol $\mathcal{Q}\stackrel{(C_n)_{n}}{\longhookrightarrow}\mathcal{R}$ means that $\mathcal{Q}\subseteq\mathcal{R}$ and $\|P\|_\mathcal{R}\le C_n\|P\|_\mathcal{Q}$ for every $P\in\mathcal{Q}(^nE;F)$. When $C_n=1$ for every $n\in\mathbb{N}$, we simply write
$\mathcal{Q}\longhookrightarrow\mathcal{R}$.

The identity operator on the Banach space $E$ shall be denoted by $id_E$. We use the standard notation from the theory of operator ideals (see \cite{klauslivro, pietschlivro}).

\section{Hyper-ideals and two-sided ideals}\label{sec2}

The next definition introduces the notions of ($p$-normed, $p$-Banach) polynomial hyper-ideals and two-sided ideals and recalls the notion of polynomial ideals. Given $\varphi \in E'$ and $y \in F$, consider the $n$-homogeneous polynomial $\varphi^n \otimes y \in {\cal P}(^nE;F)$ given by $\varphi^n \otimes y(x) = \varphi(x)^ny$. Finite sums of polynomial of the this type are called {\it polynomials of finite type}.

\begin{definition}\label{dhip}\rm Let $0<p\le1$, $\mathcal{Q}$ be a class of polynomials, $(C_n)_{n=1}^\infty$ be a sequence of positive real numbers with $C_n\ge1$ for every $n\in\mathbb{N}$ and $C_1=1$, and $(C_n,K_n)_{n=1}^\infty$ be a sequence of pairs of positive real numbers with $C_n,K_n\ge1$ for every $n\in\mathbb{N}$ and $C_1=K_1=1$. For all $n\in \mathbb{N}$ and Banach spaces $E$ and $F$, assume that: (i) the component $$\mathcal{Q}(^nE;F):=\mathcal{P}(^nE;F)\cap \mathcal{Q}$$ is a linear subspace of $\mathcal{P}(^nE;F)$ containing the $n$-homogenous polynomials of finite type, (ii) the restriction of $\|\cdot\|_{\mathcal{Q}}$ to $\mathcal{Q}(^nE;F)$ is a $p$-norm, (iii) $\|\widehat{I_n}\colon \mathbb{K} \longrightarrow\mathbb{K}~,~ \widehat{I_n}(\lambda)=\lambda^n\|_{\mathcal{Q}}=1$ for every $n$.\\
(a) We say that $\mathcal{Q}$ is a \textit{$p$-normed polynomial $(C_n)_{n=1}^\infty$-hyper-ideal} if it satisfies the\\
 {\bf Hyper-ideal property:} For $n,m \in \mathbb{N}$, and Banach spaces $E$, $F$, $G$ and $H$, if $P\in\mathcal{Q}(^nE;F)$, $Q\in\mathcal{P}(^mG;E)$ and $t\in\mathcal{L}(F;H)$, then $t\circ P\circ Q\in \mathcal{Q}(^{mn}G;H)$ and
$$\|t\circ P\circ Q\|_{\mathcal{Q}}\le C_m^n\cdot\|t\|\cdot \|P\|_{\mathcal{Q}}\cdot\|Q\|^n.$$

When $C_n=1$ for every $n\in\mathbb{N}$, we simply say that $\mathcal{Q}$ is a {\it $p$-normed polynomial hyper-ideal}. If the components $\mathcal{Q}(^nE;F)$ are complete with respect to the topology generated by $\|\cdot\|_{\mathcal{Q}}$, then $\mathcal{Q}$ is called a \textit{$p$-Banach polynomial $(C_n)_{n=1}^\infty$-hyper-ideal}. When $p=1$ we say that $\mathcal{Q}$ is a \textit{normed (Banach) polynomial $(C_n)_{n=1}^\infty$-hyper-ideal}. If  $\mathcal{Q}$ is a $p$-normed ($p$-Banach) polynomial $(C_n)_{n=1}^\infty$-hyper-ideal for some $p \in (0,1)$, then we say that it is a {\it quasi-normed (quasi-Banach) polynomial $(C_n)_{n=1}^\infty$-hyper-ideal}. When $\mathcal{Q}$ is complete with respect to the usual sup norm we say that it is a {\it closed polynomial hyper-ideal}.

\medskip

\noindent (b) When the hyper-ideal property holds for every $n \in \mathbb{N}$, but only for $m = 1$, we say that $\cal Q$ is a {\it $p$-normed polynomial ideal} (remember that $C_1=1$). The corresponding notions of $p$-Banach polynomial ideal, normed/Banach polynomial ideal, quasi-normed/quasi-Banach polynomial ideal and closed polynomial ideal are defined in the obvious way.

\medskip

\noindent (c) We say that $\mathcal{Q}$ a \textit{$p$-normed polynomial $(C_n,K_n)_{n=1}^\infty$-two-sided ideal} if it satisfies the\\
{\bf Two-sided ideal property:} For $n,m, r \in \mathbb{N}$, and Banach spaces $E$, $F$, $G$ and $H$, if  $P\in\mathcal{Q}(^nE;F)$, $Q\in\mathcal{P}(^mG;E)$ and $R\in\mathcal{P}(^rF;H)$, then $R\circ P\circ Q\in \mathcal{Q}(^{rmn}G;H)$ and $$\|R\circ P\circ Q\|_{\mathcal{Q}}\le K_r\cdot C_m^{rn}\cdot\|R\|\cdot \|P\|_{\mathcal{Q}}^r\cdot\|Q\|^{rn}.$$

When $C_n=K_n=1$, for every $n\in\mathbb{N}$ we simply say that $\mathcal{Q}$ is a {\it $p$-normed polynomial two-sided ideal}. The corresponding notions of $p$-Banach polynomial $(C_n,K_n)_{n=1}^\infty$-two-sided ideal, normed/Banach polynomial $(C_n,K_n)_{n=1}^\infty$-two-sided ideal, quasi-normed/quasi-Banach polynomial $(C_n,K_n)_{n=1}^\infty$-two-sided ideal,
and closed polynomial two-sided ideal, are defined in the obvious way.\end{definition}

\begin{obs}\label{obssmi}\rm (i) The condition $C_1 = K_1 = 1$ guarantees that every (normed, quasi-normed, Banach, quasi-Banach) polynomial $(C_n,K_n)_{n=1}^\infty$-two-sided ideal is a (normed, quasi-normed, Banach, quasi-Banach) polynomial $(C_n)_{n=1}^\infty$-hyper-ideal; and that every (normed, quasi-normed, Banach, quasi-Banach) polynomial $(C_n)_{n=1}^\infty$-hyper-ideal is a (nor-med, quasi-normed, Banach, quasi-Banach) polynomial ideal. So, properties of polynomial ideals are inherited by polynomial hyper-ideals and two-sided ideals. For instance, if $\mathcal{Q}$ is a $p$-normed polynomial $(C_n,K_n)_{n=1}^\infty$-two-sided ideal (or $p$-normed polynomial $(C_m)_{m=1}^\infty$-hyper-ideal), then \begin{equation}\|\cdot\|\le\|\cdot\|_\mathcal{Q}.\label{deship}\end{equation}

(ii) Let us see that conditions $C_n,K_n\ge1$ for every $n\in\mathbb{N}$ are not restrictive. For hyper-ideals and two-sided ideals, considering the polynomial $\widehat{I_n}:\mathbb{K}\longrightarrow\mathbb{K}$ given by $\widehat{I_n}(\lambda)=\lambda^n$, we have $$1=\|\widehat{I_n}\|_\mathcal{Q}=\|\widehat{I_1}\circ\widehat{I_n}\|_\mathcal{Q}\le C_n^1\cdot\|\widehat{I_1}\|_\mathcal{Q}\cdot\|\widehat{I_n}\|=C_n.$$
For two-sided ideals, $$1=\|\widehat{I_n}\|_\mathcal{Q}=\|\widehat{I_n}\circ\widehat{I_1}\|_\mathcal{Q}\le K_n\cdot\|\widehat{I_n}\|\cdot\|\widehat{I_1}\|_\mathcal{Q}=K_n.$$\end{obs}

Given $Q\in\mathcal{P}(^mE)$, $y\in F$ and $n \in \mathbb{N}$, consider the following $mn$-homogenous polynomial:
$$Q^n\otimes y \colon E \longrightarrow F~,~Q^n\otimes y(x)=Q(x)^n\cdot y.$$
A polynomial $P \in {\cal P}(^nE;F)$ is said to have {\it finite rank}, in symbols $P \in {\cal P}_{\cal F}(^nE;F)$, if the subspace of $F$ generated by the range $P(E)$ of $P$ is finite dimensional; or, equivalently, if $P$ is a finite sum of polynomials of the type $Q\otimes y$, where $Q \in {\cal P}(^nE)$ and $y \in F$.

\begin{prop}\label{postoum}Let $Q\in\mathcal{P}(^mE)$, $y\in F$ and $n \in \mathbb{N}$. If $\mathcal{Q}$ a $p$-normed polynomial $(C_n)_{n=1}^\infty$-hyper-ideal or a $p$-normed polynomial $(C_n,K_n)_{n=1}^\infty$-two-sided ideal, then $Q^n\otimes y \in {\cal Q}(^{mn}E;F)$ and $$\|Q^n\otimes y\|_\mathcal{Q} \le C_{m}^n\cdot \|Q\|^n\cdot\|y\|.$$ If $Q$ is a linear operator or $C_n = 1$ for every $n$, then $\|Q^n\otimes y\|_\mathcal{Q}= \|Q\|^n\cdot\|y\|$.\end{prop}

\begin{proof} Let us treat the hyper-ideal case (the two-sided ideal case is analogous). By $1 \otimes y$ we mean the operator $\lambda \in \mathbb{K} \mapsto \lambda y \in F$. Writing $Q^n\otimes y=(1\otimes y)\circ \widehat{I_n}\circ Q,$ the hyper-ideal property gives  $Q^n\otimes y \in {\cal Q}(^{mn}E;F)$ and $$\|Q^m\otimes y\|_{\cal Q}\le C_{m}^n\cdot \|1\otimes y\|\cdot \|\widehat{I_n}\|_{\mathcal{Q}}\cdot\|Q\|^n=C_m^n\cdot\|1\otimes y\|\cdot\|Q\|^n = C_m^m \cdot \|Q\|^n \cdot \|y\|.$$
For the second assertion, recall that $C_1 = 1$ and combine the inequality we have just proved with
$$\|Q\|^n\cdot\|y\|= \|Q^n\otimes y\| \leq \|Q^m\otimes y\|_{\cal Q}.$$\end{proof}

The following general properties of hyper-ideals and two-sided ideals of polynomials can also be proved by standard arguments (see, e.g. \cite{ewerton}).

\begin{prop}\label{fhip} Given a polynomial hyper-ideal (resp. two sided-ideal) $\mathcal{Q}$, define $$\overline{\mathcal{Q}}(^nE;F):=\overline{\mathcal{Q}(^nE;F)}^{\|\cdot\|},$$ for all $n\in \mathbb{N}$ and Banach spaces $E,F$.
Then $\overline{\mathcal{Q}}$ is the smallest closed polynomial hyper-ideal (resp. two sided-ideal) containing $\mathcal{Q}$.\end{prop}

\begin{teo}[Series criterion]\label{cs} Let $0<p\le 1$  and $\mathcal{Q}$ be a class of polynomials satisfying the following two conditions:\\
\indent {\rm (i)} $\widehat{I_n}\in\mathcal{Q}(^n\mathbb{K};\mathbb{K})$ and $\|\widehat{I_n}\|_{\mathcal{Q}}=1$ for every $n \in \mathbb{N}$.\\
\indent {\rm (ii)} If $(P_j)_{j=1}^\infty\subseteq\mathcal{Q}(^nE;F)$ is such that $\sum\limits_{j=1}^\infty\|P_j\|_\mathcal{Q}^{p}<\infty$, then $$P:=\sum\limits_{j=1}^\infty P_j\in\mathcal{Q}(^nE;F)\ \mbox{and}\ \|P\|_\mathcal{Q}^{p}\le\sum\limits_{j=1}^\infty \|P_j\|_\mathcal{Q}^{p}.$$
{\rm(a)} $\mathcal{Q}$ is a $p$-Banach polynomial $(C_n)_{n=1}^\infty$-hyper-ideal if and only if the hyper-ideal property (cf. Definition \ref{dhip}(a)) holds.\\
{\rm (b)} $\mathcal{Q}$ is a $p$-Banach polynomial $(C_n,K_n)_{n=1}^\infty$-two-sided ideal if and only if the two-sided ideal property (cf. Definition \ref{dhip}(c)) holds.
\end{teo}

\begin{prop}\label{phipdf}Let $0 < p,q \leq 1$, $\mathcal{Q}$ be a $p$-Banach polynomial $(C_n)_{n=1}^\infty$-hyper-ideal ($(C_n,K_n)_{n=1}^\infty$-two-sided ideal, resp.) and $\mathcal{R}$ be a $q$-Banach polynomial $(D_n)_{n=1}^\infty$-hyper-ideal ($(D_n,L_n)_{n=1}^\infty$-two-sided ideal, resp.). If $\mathcal{Q}\subseteq\mathcal{R}$, then there are positive numbers $(M_n)_{n=1}^\infty$ such that $\mathcal{Q}\stackrel{(M_n)_{n}}{\longhookrightarrow}\mathcal{R}$.\end{prop}

\section{Distinguished examples}
\label{sec:mfatat}
This section has a twofold purpose: (i) to give illustrative examples of polynomial hyper-ideals and of polynomial two-sided-ideals; (ii) to give interesting examples of polynomial ideals that fail to be hyper-ideals and interesting examples of polynomial hyper-ideals that fail to be two-sided ideals. The main idea is to make clear that it is worth studying the two new classes, namely, polynomial hyper-ideals and polynomial two-sided-ideals.

\begin{ex}\label{ex1p}\rm The polynomial ideal ${\cal P}_f$ of finite type polynomials is not a hyper-ideal (hence it is not a two-sided-ideal). Indeed, it is well known that the 2-homogeneous polynomial
$$P \colon \ell_2\longrightarrow\mathbb{K}~,~P(x)=\sum\limits_{j=1}^\infty x_j^2,$$  is not of finite type (see, e.g. \cite[Example 3.1]{ewerton}). Then $\mathcal{P}_f$ cannot be a polynomial hyper-ideal, because otherwise we would have $P=id_{\mathbb{K}}\circ P\in\mathcal{P}_f(^2 \ell_2)$.
\end{ex}

Once the polynomial ideal of finite type polynomials is out of the game, the class of finite rank polynomials is the natural candidate to be the smallest polynomial hyper-ideal/two-sided ideal. A class of polynomials to which no specific norm has been assigned is supposed to be endowed with the usual sup norm.

\begin{teo}\label{postofinmenor} The class $\mathcal{P}_\mathcal{F}$ of finite rank homogenous polynomials is the smallest polynomial hyper-ideal and the smallest polynomial two-sided ideal, meaning that $\mathcal{P}_\mathcal{F}$ is a polynomial two-sided ideal (hence a polynomial hyper-ideal) and if $\mathcal{Q}$ is a polynomial hyper-ideal or a polynomial two-sided ideal, then $\mathcal{P}_\mathcal{F}\subseteq\mathcal{Q}$.\end{teo}

\begin{proof}Since ${\cal P}_{\cal F}$ is a polynomial ideal, to prove that $\mathcal{P}_\mathcal{F}$ is a polynomial two-sided ideal (in particular a hyper-ideal), we just need check that $R\circ P\circ Q\in\mathcal{P}_\mathcal{F}(^{mnr}G;H)$ whenever $P\in\mathcal{P}_\mathcal{F}(^nE;F)$, $Q\in \mathcal{P}(^mG;E)$ and $R\in\mathcal{P}(^rF;H)$. It is clear that
$${\rm span}\{R \circ P \circ Q(H)\} \subseteq {\rm span}\{R \circ P(E)\} \subseteq {\rm span}\{R({\rm span}\{P(E)\})\}. $$
Since ${\rm span}\{P(E)\}$ is finite-dimensional and homogeneous polynomials on finite-dimensional spaces have finite rank, it follows that $R \circ P \circ Q$ has finite rank. From Proposition \ref{postoum} we know that $\cal Q$ contains all rank 1 homogeneous polynomials, that is, polynomials of the type $Q \otimes y$, where $Q$ is a scalar-valued homogeneous polynomial. Any finite rank polynomial is a finite sum of rank 1 polynomials, hence ${\cal P}_{\cal F} \subseteq {\cal Q}$.\end{proof}

In particular, in the definitions of polynomial hyper-ideal and polynomial two-sided ideal, the containment of the finite type polynomials is equivalent to the containment of the finite rank polynomials:

\begin{cor}\label{coreqhi}Let $\mathcal{Q}$ be a class of continuous homogenous polynomials fulfilling the hyper-ideal property (cf. Definition \ref{dhip}(a)) or the two-sided ideal property (cf. Definition \ref{dhip}(c)) such that each component $\mathcal{Q}(^nE;F)$ is a linear subspace of $\mathcal{P}(^nE;F)$. Then ${\cal P}_f \subseteq \cal Q$ if and only if ${\cal P}_{\cal F} \subseteq \cal Q$.\end{cor}

\begin{ex}\label{exapp} \rm It is not difficult to check that the $P$ from Example \ref{ex1p} is cannot be approximated by finite type polynomials (cf. \cite[Example 3.4]{ewerton}). By Theorem \ref{postofinmenor} it follows that the class of polynomials that can be approximated, in the uniform norm, by finite type polynomials is not a polynomial hyper-ideal.\end{ex}

Later in this section we shall identify the smallest Banach polynomial hyper/two-sided ideal. Now we can exhibit the smallest closed polynomial hyper/two-sided ideal:

\begin{ex}\rm Combining Proposition \ref{fhip} with Theorem \ref{postofinmenor}, we conclude that the class $\overline{\mathcal{P}_\mathcal{F}}$ of polynomials that can be approximated, in the uniform norm, by finite rank polynomials is the smallest closed polynomial two-sided ideal.\end{ex}

Next we show that an example of a Banach polynomial hyper-ideal that fails to be a two-side-ideal can be found within the classical classes of polynomials:

\begin{ex}\label{expsi}\rm On the one hand, the class $\mathcal{P}_\mathcal{K}$ of compact homogeneous polynomials is closed polynomial two-sided ideal. On the other hand, the class $\mathcal{P}_\mathcal{W}$ of weakly compact homogeneous polynomials is a closed polynomial hyper-ideal that fails to be a polynomial two-sided ideal.

The case of $\mathcal{P}_\mathcal{K}$ follows from the fact that continuous homogeneous polynomials send bounded sets to bounded sets and compact sets to compact sets.

That $\mathcal{P}_\mathcal{W}$ is a polynomial hyper-ideal follows from the following two facts: (i) continuous homogeneous polynomials send bounded sets to bounded sets; (ii) bounded linear operators are weak-to-weak continuous. Alternatively, see Examples \ref{exexex} and \ref{excfc}.

For $n \geq 2$, consider the $n$-homogeneous polynomial $$P\colon \ell_1\longrightarrow\ell_1~, P((\lambda_j)_{j=1}^\infty)=(\lambda_j^n)_{j=1}^\infty.$$
The polynomial $P$ is not weakly compact because the sequence $(P(e_j))_{j=1}^\infty = (e_j)_{j=1}^\infty $ has no weakly convergent subsequence in $\ell_1$. Write $P=Q\circ \iota,$ were $\iota \colon \ell_1 \hookrightarrow\ell_2$ is the formal inclusion and $Q\in\mathcal{P}(^n\ell_2;\ell_1)$ is given by the same rule as $P$. The reflexivity of $\ell_2$ gives $\iota\in\mathcal{W}(\ell_1;\ell_2)=\mathcal{P}_\mathcal{W}(^1\ell_1;\ell_2)$, so  $\mathcal{P}_\mathcal{W}$ is not a polynomial two-sided ideal, because otherwise we would have $P=Q\circ \iota\in\mathcal{P}_\mathcal{W}(^n\ell_1;\ell_1)$. Further details can be found in \cite[Example 27]{botelho1}.\end{ex}

The example above makes clear that, although the notion of polynomial two-sided ideal seems to be the {\it right} concept, the intermediate notion of hyper-ideal is also worth being studied, because otherwise the hyper-ideal property of important classes, such as the class of weakly compact polynomials, would go unnoticed.

Our next purpose is to give an explicit description of the smallest Banach polynomial hyper/two-sided ideal. First we show that the smallest Banach polynomial ideal is helpless:

\begin{ex}\rm Recall that an $n$-homogenous polynomial $P\in\mathcal{P}(^nE;F)$ is {\it nuclear} if it admits a nuclear representation
$$P(x)=\sum\limits_{j=1}^\infty \lambda_j(\varphi_j(x))^n\cdot y_j {\rm ~for~ every~} x \in E,$$
where $(\varphi_j)_{j=1}^\infty$ is a bounded sequence in $E'$,  $(\lambda_j)_{j=1}^\infty\in\ell_1$ and $(y_j)_{j=1}^\infty$ is a bounded sequence in $F$. In the space $ \mathcal{P}_\mathcal{N}(^nE;F)$ of all nuclear $n$-homogenous polynomial from $E$ to $F$ define
$$\|P\|_{\mathcal{P}_\mathcal{N}}=\inf\left\{\sum\limits_{j=1}^\infty |\lambda_j|\cdot\|\varphi_j\|^n\cdot\|y_j\|\right\},$$ where the infimum is taken over all nuclear representations of $P$. It is well known that ${\cal P}_{\cal N}$ is the smallest Banach polynomial ideal. Let us see that ${\cal P}_{\cal N}$
fails to be a polynomial hyper/two-sided ideal. In fact, considering the partial sums of a nuclear representation of a nuclear polynomial, it is not difficult to check that $\mathcal{P}_\mathcal{N}\subseteq \overline{{\cal P}_f}$. In Example \ref{exapp} we saw that $\overline{{\cal P}_f}$ does not contain ${\cal P}_{\cal F}$, so $\mathcal{P}_\mathcal{N}$ does not contain ${\cal P}_{\cal F}$ either. From Theorem \ref{postofinmenor} we conclude that $\mathcal{P}_\mathcal{N}$ does not fulfil the hyper-ideal property.\end{ex}

Once the smallest Banach polynomial ideal is out of the game, we need another class to play the role of the smallest Banach polynomial hyper/two-sided ideal. To do so we adapt for polynomial the class of hyper-nuclear multilinear operators introduced in \cite{ewerton}. Henceforth we assume $1/\infty = 0$. By $\ell_r(E)$ and $\ell_r^w(E)$ we denote the spaces os absolutely $r$-summable and weakly $r$-summable $E$-valued sequences, endowed with their usual norms ($r$-norms if $0<r < 1$) $\|\cdot\|_r$ and $\|\cdot\|_{w,r}$.

\begin{definition}\rm Let $s\in(0,\infty)$ and $r\in [1,\infty]$ be such that $1\le1/s+1/r$. An $n$-homogenous polynomial $P \in \mathcal{P}(^nE;F)$ is called:\\
(i) \textit{hyper-$(s,r)$-nuclear} if there are sequences $(\lambda_j)_{j=1}^\infty\in\ell_s$, $(P_j)_{j=1}^\infty\in\ell_r^w(\mathcal{P}(^nE))$ and $(y_j)_{j=1}^\infty\in\ell_\infty(F)$ such that \begin{equation}\label{eq:hn}P(x)=\sum\limits_{j=1}^\infty\lambda_jP_j\otimes y_j(x)=\sum\limits_{j=1}^\infty\lambda_jP_j(x) y_j,\end{equation} for every $x\in E$. In this case we write $P\in\mathcal{P}_{\mathcal{HN}_{(s,r)}}(^nE;F)$ and define $$\|P\|_{\mathcal{P}_{\mathcal{HN}_{(s,r)}}}=\inf\{\|(\lambda_j)_{j=1}^\infty\|_s\cdot \|(P_j)_{j=1}^\infty\|_{w,r}\cdot\|(y_j)_{j=1}^\infty\|_\infty\},$$ where the infimum is taken over all representations of $P$ as in (\ref{eq:hn}).\\
(ii) \textit{strong-$(s,r)$-nuclear} if there are sequences $(\lambda_j)_{j=1}^\infty\in\ell_s$, $(P_j)_{j=1}^\infty\in\ell_r(\mathcal{P}(^nE))$ and $(y_j)_{j=1}^\infty\in\ell_\infty(F)$ such that \begin{equation}\label{eq:sn}P(x)=\sum\limits_{j=1}^\infty\lambda_jP_j\otimes y_j(x)=\sum\limits_{j=1}^\infty\lambda_jP_j(x) y_j,\end{equation} for every $x\in E$. In this case we write $P\in\mathcal{P}_{\mathcal{SN}_{(s,r)}}(^nE;F)$ and define $$\|P\|_{\mathcal{P}_{\mathcal{SN}_{(s,r)}}}=\inf\{\|(\lambda_j)_{j=1}^\infty\|_s\cdot \|(P_j)_{j=1}^\infty\|_{r}\cdot\|(y_j)_{j=1}^\infty\|_\infty\},$$ where the infimum is taken over all representations of $P$ as in (\ref{eq:sn}).\end{definition}

It is plain that $\mathcal{P}_{\mathcal{SN}_{(s,r)}}\longhookrightarrow\mathcal{P}_{\mathcal{HN}_{(s,r)}}$. When $s=1$ and $r=\infty$ we simply write $\mathcal{P}_{\mathcal{HN}}$ and its elements are called hyper-nuclear polynomials.

\begin{teo}Let $s\in(0,\infty)$, $r\in [1,\infty]$ be such that $1\le1/s+1/r$ and $p$ be such that $\frac{1}{p}=\frac{1}{s}+\frac{1}{r}$.\\
{\rm (a)} The class $\mathcal{P}_{\mathcal{HN}_{(s,r)}}$ of hyper-$(s,r)$-nuclear homogenous polynomials is a $p$-Banach hyper-ideal.\\
{\rm (b)} The class $\mathcal{P}_{\mathcal{SN}_{(s,r)}}$ of strong-$(s,r)$-nuclear homogenous polynomials is a $p$-Banach $\ap1,\frac{n^n}{n!}\fp_{n=1}^\infty$-two-sided ideal. \end{teo}

\begin{proof}(a) The proof, using the series criterion (Theorem \ref{cs}(a)), follows from an obvious adaptation of the proof of \cite[Theorem 3.6]{ewerton}.\\
(b) To prove the two-sided ideal property, let $P\in\mathcal{P}_{\mathcal{SN}_{(s,r)}}(^nE;F)$, $Q\in \mathcal{P}(^mG;E)$ and $R\in\mathcal{P}(^lF;H)$. We can write $P=\sum\limits_{j=1}^\infty\lambda_jP_j\otimes y_j$, where $(\lambda_j)_{j=1}^\infty\in\ell_s$, $(P_j)_{j=1}^\infty\in\ell_r(\mathcal{P}(^nE))$ and $(y_j)_{j=1}^\infty\in\ell_\infty(F)$. For $j_1, \ldots, j_l \in \mathbb{N}$, $(P_{j_1} \cdots P_{j_l}) \circ Q \in {\cal P}(^{nml} G)$ ($x \in G \mapsto P_{j_1}(Q(x)) \cdots P_{j_l}(Q(x)) \in \mathbb{K}$). From
\begin{align*}
\sum\limits_{j_1,\ldots,j_l=1}^\infty \|(P_{j_1}\cdots P_{j_l})\circ Q\|^r&\le \sum\limits_{j_1,\ldots,j_l=1}^\infty \|P_{j_1}\|^r\cdots \|P_{j_l}\|^r\cdot\|Q\|^{nlr}\\
&= \left(\sum\limits_{j_1=1}^\infty\cdots\sum\limits_{j_l=1}^\infty\|P_{j_1}\|^r\cdots \|P_{j_l}\|^r\right)\cdot\|Q\|^{nlr}\\
&=\ap\sum\limits_{j_1=1}^\infty\|P_{j_1}\|^r\fp\cdots \ap\sum\limits_{j_l=1}^\infty\|P_{j_l}\|^r\fp\cdot\|Q\|^{nlr}\\
&=\ap\sum\limits_{j=1}^\infty\|P_{j}\|^r\fp^l\cdot\|Q\|^{nlr},
\end{align*} we deduce that $((P_{j_1}\cdots P_{j_l})\circ Q)_{j_1,\ldots,j_l=1}^\infty\in\ell_r(\mathcal{P}(^{nml}G))$ and $$\|((P_{j_1}\cdots P_{j_l})\circ Q)_{j_1,\ldots,j_l=1}^\infty\|_r\le\|(P_j)_{j=1}^\infty\|_r^l\cdot\|Q\|^{nl}.$$
Analogously, $(\lambda_{j_1}\cdots \lambda_{j_l})_{j_1,\ldots,j_l=1}^\infty\in\ell_s$ with $\|(\lambda_{j_1}\cdots \lambda_{j_l})_{j_1,\ldots,j_l=1}^\infty\|_s=\|(\lambda_j)_{j=1}^\infty\|_s^l$. We also have $$\|\check{R}(y_{j_1},\ldots,y_{j_l})\|\le\|\check{R}\|\cdot\|y_{j_1}\|\cdots\|y_{j_l}
\|\le\dfrac{l^l}{l!}\|R\|\cdot\|(y_j)_{j=1}^\infty\|_\infty^l,$$ which gives $(\check{R}(y_{j_1},\ldots,y_{j_l}))_{j_1,\ldots,j_l=1}^\infty\in\ell_\infty(H)$ with $$\|(\check{R}(y_{j_1},\ldots,y_{j_l}))_{j_1,\ldots,j_l=1}^\infty\|_\infty\le\dfrac{l^l}{l!} \|R\|\cdot\|(y_j)_{j=1}^\infty\|_\infty^l.$$
Since $1\le1/s+1/r$, H\"older's inequality shows that, for every $x \in G$, the series
\begin{equation}\label{nser}\sum\limits_{j_1,\ldots,j_l=1}^\infty(\lambda_{j_1}\cdots \lambda_{j_l})(P_{j_1}\cdots P_{j_r})\circ Q\otimes \check{R}(y_{j_1},\ldots,y_{j_l})(x)
\end{equation}
is absolutely convergent, hence convergent in $H$. Using that $P=\sum\limits_{j=1}^\infty\lambda_jP_j\otimes y_j$, it is immediate that the series (\ref{nser}) converges to $R\circ P \circ Q(x)$, therefore $$\sum\limits_{j_1,\ldots,j_l=1}^\infty(\lambda_{j_1}\cdots \lambda_{j_l})(P_{j_1}\cdots P_{j_r})\circ Q\otimes \check{R}(y_{j_1},\ldots,y_{j_l})$$
is a strong-$(s,r)$-nuclear representation of $R \circ P \circ Q$, that is, $R\circ P\circ Q\in \mathcal{P}_{\mathcal{SN}_{(s,r)}}(^{nml}G:H)$. Finally,
\begin{eqnarray*}\|R\circ P\circ Q\|_{\mathcal{P}_{\mathcal{SN}_{(s,r)}}}&\le& \|((P_{j_1}\cdots P_{j_l})\circ Q)_{j_1,\ldots,j_l=1}^\infty\|_r\cdot\|(\lambda_{j_1}\cdots \lambda_{j_l})_{j_1,\ldots,j_l=1}^\infty\|_s \cdot\\& &\|\check{R}(y_{j_1},\ldots,y_{j_l})_{j_1,\ldots,j_l=1}^\infty\|_\infty\\
&\leq&\dfrac{l^l}{l!}\|R\|\cdot \left(\|(P_j)_{j=1}^\infty\|_r\cdot\|(\lambda_j)\|_s\cdot\|(y_j)_{j=1}^\infty\|_\infty\right)^l
\cdot\|Q\|^{nl}.\end{eqnarray*} The desired inequality follows by taking the infimum over all representations of $P$.
\end{proof}

In particular, taking $r=1$ and $s=\infty$ we get that the class of hyper-nuclear polynomials is a Banach two-sided ideal. Moreover, reasoning as in the proof of \cite[Theorem 3.10]{ewerton}, we have the following:

\begin{teo}\label{snmb}The class $\mathcal{P}_{\mathcal{HN}}$ of hyper-nuclear homogenous polynomials is the smallest Banach $\ap1,\frac{n^n}{n!}\fp_{n=1}^\infty$-two-sided ideal (resp. Banach hyper-ideal), in sense that if $\mathcal{Q}$ is a Banach $(C_n,K_n)_{n=1}^\infty$-two-sided ideal
(resp. Banach $(C_n)_{n=1}^\infty$-hyper-ideal), then $\mathcal{P}_{\mathcal{HN}}\stackrel{(C_n)_{n}}{\longhookrightarrow}\mathcal{Q}$.\end{teo}

Next we see another example of a classical closed polynomial ideal that fails to be a hyper-ideal:

\begin{ex}\rm Let $\mathcal{P}_{wsc}$ denote the class of weakly sequentially continuous homogeneous polynomials, that is, continuous homogeneous polynomials that send weakly convergent sequences to norm convergent sequences. It is well known that $\mathcal{P}_{wsc}$ is a closed polynomial ideal. Let us see that $\mathcal{P}_{wsc}$ is not a polynomial hyper/two-sided-ideal. Let $$P\colon \ell_2\longrightarrow\ell_1~,~ P((\lambda_j)_{j=1}^\infty)=(\lambda_j^2)_{j=1}^\infty.$$ The sequence $(e_j)_{j=1}^\infty$ of canonical unit vectors is weakly null in $\ell_2$ and non-norm null in $\ell_1$, so $P\notin \mathcal{P}_{wsc}(^2\ell_2;\ell_1)$. Since  $id_{\ell_1}\in\mathcal{P}_{wsc}(\ell_1;\ell_1)$ and $P=id_{\ell_1}\circ P$, $\mathcal{P}_{wsc}$ does not satisfy the hyper-ideal property.\end{ex}

Like the case of other polynomial ideals that fail to be hyper/two-sided ideals, we can provide a class that, in a certain sense, plays the role of ${\cal P}_{wsc}$ in the context of hyper/two-sided ideals. To do so we use the notion of polynomial convergence (see, for example, \cite[Section 6]{colegamelin}). By a {\it scalar-valued polynomial} on $E$ we mean a finite sum of scalar-valued homogeneous polynomials on $E$, in general of different degrees of homogeneity.

\begin{definition}\rm (a) A sequence $(x_j)_{j=1}^\infty$ in a Banach space $E$ {\it converges polynomially} to $x \in E$ if $p(x_j)\longrightarrow p(x)$ for every scalar-valued polynomial $p$ on $E$.\\
(b) A polynomial $P\in \mathcal{P}(^nE;F)$ is said to be {\it polynomially sequentially continuous} if $P(x_j)\longrightarrow P(x)$ in $F$ whenever $(x_j)_{j=1}^\infty$ converges polynomially to $x$ in $E$.\end{definition}

\begin{prop}\label{proppsc} The class $\mathcal{P}_{psc}$ of polynomially sequentially continuous polynomials is a closed polynomial two-sided ideal.\end{prop}

\begin{proof} It is plain that $\mathcal{P}_{psc}(^nE;F)$ is linear subspace of ${\cal P}(^nE;F)$ containing the finite rank polynomials. As to the two-sided ideal property, let $Q\in \mathcal{P}(^mG;E)$, $P \in \mathcal{P}_{psc}(^nE;F)$, $R \in \mathcal{P}(^rF;H)$ and $(x_j)_{j=1}^\infty$ be a sequence in $G$ converging polynomially to $x$. For every scalar-valued polynomial $p$ on $E$, we have that $p\circ Q$ is a scalar-valued polynomial on $G$, so $p(Q(x_j))\longrightarrow p(Q(x)),$ showing that $(Q(x_j))_{j=1}^\infty$ converges polynomially to $Q(x)$. Since $P$ is polynomially sequentially continuous, $P(Q(x_j))\longrightarrow P(Q(x))$ in $F$. The continuity of $R$ gives $R\circ P \circ Q(x_j) \longrightarrow R\circ P\circ Q(x)$ in $H$, proving that $R\circ P\circ Q\in \mathcal{P}_{psc}(^{rnm}G;H).$ By definition, polynomially convergent sequences are weakly convergent, hence bounded. This allows us to follow the steps of the proof of \cite[Theorem 4.8]{ewerton} to conclude that $\mathcal{P}_{psc}(^nE;F)$ is closed in $\mathcal{P}(^nE;F)$.\end{proof}

We shall return to the closed two-sided ideal $\mathcal{P}_{psc}$ in Example \ref{exex}.

\section{Polynomial hyper-ideals and two-sided ideals generated by multilinear hyper-ideals}
\label{multil}
We start a series of three sections in which we provide methods to generate polynomial hyper/two-sided ideals. The first method generates polynomial hyper/two-sided ideals from a given hyper-ideal of multilinear operators. We have so far referred to the notion of hyper-ideal of multilinear operators several times, but at this point we need the definition in detail. We use the usual notation: ${\cal L}(E_1, \ldots, E_n;F)$ denotes the Banach space of continuous $n$-linear operators from $E_1 \times \cdots \times E_n$ to $F$ endowed with the usual sup norm. If $E_1 = \cdots = E_n = E$, we write ${\cal L}(^nE;F)$. For $A \in {\cal L}(^nE;F)$, by $\widehat{A}$ we denote the $n$-homogeneous polynomial given by $\widehat{A}(x)= A(x, \ldots, x)$. And for $P \in {\cal P}(^nE;F)$, by $\check P$ we mean the unique symmetric $n$-linear operator in ${\cal L}(^nE;F)$ such that $(\check P)^\wedge= P$.

\begin{definition}\label{dhi}\rm Let $p\in (0,1]$. A \textit{$p$-normed hyper-ideal of multilinear operators}, or simply a \textit{$p$-normed multilinear hyper-ideal}, is a subclass $\mathcal{H}$ of the class of all continuous multilinear operators between Banach spaces endowed with a function $\|\cdot\|_{\mathcal{H}} \colon \mathcal{H} \longrightarrow [0,\infty)$ such that for all $n\in \mathbb{N}$ and Banach spaces $E_1, \ldots, E_n$ and $F$, the components $$\mathcal{H}(E_1,\ldots, E_n;F):=\mathcal{L}(E_1,\ldots, E_n;F)\cap \mathcal{H}$$ satisfy:\\
$(1)$ $\mathcal{H}(E_1,\ldots, E_n;F)$ is a linear subspace of $\mathcal{L}(E_1,\ldots, E_n;F)$ which contains the $n$-linear operators of finite type, on which $\|\cdot\|_{\mathcal{H}}$ is a $p$-norm.\\
(2) $\|I_n \colon \mathbb{K}^n\longrightarrow \mathbb{K}, I_n(\lambda_1,\ldots,\lambda_n)=\lambda_1\cdots\lambda_n\|_{\cal H} =1$ for every $n$.\\
$(3)$ {\bf The hyper-ideal property:} Given natural numbers $n$ and $1\le m_1<\cdots<m_n$, Banach spaces $G_1,\ldots,G_{m_n}$, $E_1,\ldots,E_n$, $F,H$, if  $B_1\in \mathcal{L}(G_1,\ldots, G_{m_1};E_1), \ldots, B_n\in \mathcal{L}(G_{m_{n-1}+1},\ldots, G_{m_n};E_n)$, $t \in \mathcal{L}(F;H)$ and
$A \in \mathcal{H}(E_1,\ldots, E_n;F)$, then $t\circ A\circ(B_1,\ldots,B_n)$ belongs to $\mathcal{H}(G_1,\ldots, G_{m_n};H)$ and
\begin{equation}\|t\circ A\circ(B_1,\ldots,B_n)\|_{\mathcal{H}}\le\|t\|\cdot\|A\|_{\mathcal{H}}\cdot
\|B_1\|\cdots\|B_n\|,\label{eqhi}\end{equation}
The notions of \textit{$p$-Banach multilinear hyper-ideal}, \textit{normed multilinear hyper-ideal}, \textit{Banach multilinear hyper-ideal}, {\it quasi-normed multilinear hyper-ideal} and  {\it quasi-Banach multilinear hyper-ideal} are defined in the obvious way.\end{definition}

By $\cal P$ we denote the class of all continuous homogenous polynomials between Banach spaces.

\begin{definition}\rm Let $\mathcal{G}$ be a subclass of the class of all continuous multilinear operators between Banach spaces endowed with a function $\|\cdot\|_\mathcal{G}\colon \mathcal{G}\longrightarrow\mathbb{R}$. We define \\
(a) $\mathcal{P}^{\mathcal{G}}:=\{P\in\mathcal{P} : \ \mbox{there is} \ A\in\mathcal{G} \ \mbox{such that} \ P=\widehat{A}\},$ and
$$\|P\|_{\mathcal{P}^{\mathcal{G}}}=\inf\{\|A\|_{\mathcal{G}} : \ A \in \mathcal{G} \ \mbox{and} \ P=\widehat{A}\}.$$
(b) $\mathcal{P}_\mathcal{G}:=\{P\in\mathcal{P} : \check{P}\in\mathcal{G}\}$ and $\|P\|_{\mathcal{P}_\mathcal{G}}=\|\check{P}\|_\mathcal{G}.$ \\
\indent Note that $\mathcal{P}_\mathcal{G}\longhookrightarrow\mathcal{P}^\mathcal{G}$.\end{definition}

\begin{teo}\label{hiph} If $\mathcal{H}$ is a $p$-normed ($p$-Banach) multilinear hyper-ideal, then $\mathcal{P}^{\mathcal{H}}$ is a $p$-normed ($p$-Banach) polynomial $\ap\frac{n^n}{n!}\fp_{n=1}^\infty$-hyper-ideal.\end{teo}

\begin{proof} It is clear that $\mathcal{P}^\mathcal{H}(^nE;F)$ is a linear subspace of $\mathcal{P}(^nE;F)$. Using that $\|\cdot\|_{\mathcal{H}}$ is a $p$-norm on the componentes of $\cal H$, it is easy to check that $\|\cdot\|_{\mathcal{P}^\mathcal{H}}$ is a $p$-norm on the components ${\cal P}^{\cal H}$. Given a polynomial of finite type $P =\sum\limits_{j=1}^k\varphi_j^n\otimes y_j\in\mathcal{P}(^nE;F)$, $\check{P}=\sum\limits_{j=1}^k\varphi_j\otimes\stackrel{(n)}{\cdots}\otimes\varphi_j$ is an $n$-linear operator of finite type, hence $ \check P \in \cal H$. It follows that $P\in\mathcal{P}^\mathcal{H}(^nE;F)$.

It is immediate that $\|\widehat{I_n}\|_{\mathcal{P}^\mathcal{H}}\le\|I_n\|_\mathcal{H}=1.$ Supposing $\|\widehat{I_n}\|_{\mathcal{P}^\mathcal{H}}<1$, there would exist an $n$-linear form $A\in\mathcal{H}(^n\mathbb{K})$ such that $A(\lambda_1, \ldots, \lambda_n)=\lambda_1\cdots\lambda_n$ for every $\lambda_1,\ldots,\lambda_n \in \mathbb{K}$ and $\|A\|_\mathcal{H}<1$. In this case, $$1=|A(1,\ldots,1)|\le\|A\|\le\|A\|_\mathcal{H}<1,$$
a contradiction that gives $\|\widehat{I_n}\|_{\mathcal{P}^\mathcal{H}}=1$.

Given $t\in \mathcal{L}(F;H)$, $Q\in\mathcal{P}^\mathcal{H}(^nE;F)$ and $R\in\mathcal{P}(^mG;E)$, pick $A\in\mathcal{H}(^nE;F)$ such that $\widehat{A}=Q$ and $B\in\mathcal{L}(^mG;E)$ such that $\widehat{B}=R$. Then $t\circ A\circ (B,\ldots,B)\in\mathcal{H}(^{mn}G;H)$. Since $(t\circ A\circ (B,\ldots,B))^\wedge=t\circ Q \circ R$,
we conclude that $t\circ Q \circ R\in\mathcal{P}^\mathcal{H}(^{mn}G;H)$.
Furthermore,
\begin{align*}\|t\circ Q \circ R\|_{\mathcal{P}^\mathcal{H}}&\le\|t\circ A\circ (B,\ldots,B)\|_\mathcal{H}\le\|t\|\cdot\|A\|_\mathcal{H}\cdot\|B\|^n\\&\le \|t\|\cdot\|A\|_\mathcal{H}\cdot \left(\frac{m^m}{m!}\|\widehat{B}\|\right)^n =\ap\frac{m^m}{m!}\fp^n\|t\|\cdot\|A\|_\mathcal{H}\cdot\|R\|^n.\end{align*}
It follows that $\|t\circ Q \circ R\|_{\mathcal{P}^\mathcal{H}}\le \ap\frac{m^m}{m!}\fp^n\|t\|\cdot\|Q\|_\mathcal{H}\cdot\|R\|^n$.

Assume now that the multilinear hyper-ideal $\mathcal{H}$ is complete. Let $(P_j)_{j=1}^\infty$ in $\mathcal{P}^\mathcal{H}(^nE;F)$ be such that $\sum\limits_{j=1}^\infty \|P_j\|_{\mathcal{P}^\mathcal{H}}^p<\infty$. Given  $\varepsilon>0$, for each $j\in\mathbb{N}$ there is a $n$-linear operator $A_j\in\mathcal{H}(^nE;F)$ such that $\widehat{A}_j=P_j$ and  $\|A_j\|_{\mathcal{H}}<(1+\varepsilon)\|P_j\|_{\mathcal{P}^\mathcal{H}}.$ Since $$\sum\limits_{j=1}^\infty \|A_j\|_{\mathcal{H}}^p <(1+\varepsilon)^p\cdot\sum\limits_{j=1}^\infty \|P_j\|_{\mathcal{P}^\mathcal{H}}^p<\infty,$$ and $\mathcal{H}$ is complete, the series $\sum\limits_{j=1}^\infty A_j$ converges in $\mathcal{H}(^nE;F)$, say $A =\sum\limits_{j=1}^\infty A_j \in\mathcal{H}(^nE;F)$. Then $P:=\widehat{A} \in {\cal P}(^nE;F)$ and by the definition of ${\cal P}^{\cal H}$ it follows that $P\in\mathcal{P}^\mathcal{H}(^nE;F)$. Therefore,
$$\left\|P - \sum_{j=1}^k P_j\right\|_{\mathcal{P}^\mathcal{H}} =\left\|\widehat{A} - \sum_{j=1}^k \widehat{A}_j\right\|_{\mathcal{P}^\mathcal{H}} = \left\|\left(A - \sum_{j=1}^k A_j\right)^{\wedge}\right\|_{\mathcal{P}^\mathcal{H}} \le \left\|A - \sum_{j=1}^k A_j\right\|_{\mathcal{H}}\stackrel{k \rightarrow \infty}{\longrightarrow} 0,$$
which proves that $\sum\limits_{j=1}^\infty P_j = P \in \mathcal{P}^\mathcal{H}(^nE;F)$ and finishes the proof that $\mathcal{P}^\mathcal{H}$ is complete.\end{proof}

Contrary to the case of multilinear/polynomial ideals, it is not to be expected that ${\cal P}_{\cal H}$ is a polynomial hyper-ideal whenever $\cal H$ is a multilinear hyper-ideal. Let us give a brief reasoning. Given a multilinear hyper-ideal $\cal H$, $P \in {\cal P}_{\cal H}(^nE;F)$ and $Q \in {\cal P}(^m G;E)$, $P \circ Q$ does not belong to ${\cal P}_{\cal H}$ in general. Indeed, all we know is that $\check P \in {\cal H}(^nE;F)$, and from the hyper-ideal property of $\cal H$ we can only conclude that $\check P \circ (\check Q, \ldots, \check Q)$ belongs to $\cal H$. For $P \circ Q$ to belong to ${\cal P}_{\cal H}$ we should have $(P \circ Q)^\vee$ in $\cal H$, so everything would work if $(P \circ Q)^\vee = \check P \circ (\check Q, \ldots, \check Q)$. There is no hope for this equality to hold, because the multilinear operator $\check P \circ (\check Q, \ldots, \check Q)$ is not symmetric in general.

Thus we need to impose an extra condition on the multilinear hyper-ideal $\cal H$ to guarantee that ${\cal P}_{\cal H}$ is a polynomial hyper-ideal.

\begin{definition}\rm Let $S_n$ denote the group of permutations of $\{1,\ldots,n\}$. Given a multilinear operator $A\in\mathcal{L}(^nE;F)$ and a $\sigma\in S_n$, define $A_\sigma$ and  $A_s$ in $\mathcal{L}(^nE;F)$ by $$A_\sigma(x_1,\ldots,x_n)=A(x_{\sigma(1)},\ldots,x_{\sigma(n)}),$$
$$A_s(x_1,\ldots,x_n)=\frac{1}{n!}\cdot\sum\limits_{\sigma\in S_n}A(x_{\sigma(1)},\ldots,x_{\sigma(n)})=\frac{1}{n!}\cdot\sum\limits_{\sigma\in S_n}A_\sigma(x_1,\ldots,x_n).$$
\indent We say that a subclass $\mathcal{G}$ of the class of continuous multilinear operators between Banach spaces is {\it symmetric} if $A_s\in\mathcal{G}(^nE;F)$ whenever $A\in\mathcal{G}(^nE;F).$ If $\cal G$ is endowed with a function $\|\cdot\|_\mathcal{G}\colon \mathcal{G}\longrightarrow[0,\infty)$, we say that $\mathcal{G}$ is {\it strongly symmetric} if $A_\sigma \in {\cal G} {\rm ~and~} \|A_\sigma\|_\mathcal{G}=\|A\|_\mathcal{G}$ for all $A\in\mathcal{G}(^nE;F)$ and $\sigma\in S_n$.\end{definition}

The next lemma is a straightforward consequence of the following fact: if $P$ is a homogeneous polynomial and $A$ is a multilinear operator such that $P = \widehat{A}$, then
$$\check{P}=(\widehat{A})^\vee=A_s.$$

\begin{lema}\label{lemasim}A subclass $\mathcal{G}$ of the class of continuous multilinear operators between Banach spaces is symmetric if and only if  $\mathcal{P}^\mathcal{G}=\mathcal{P}_\mathcal{G}$.\end{lema}

\begin{prop}\label{corhip}
{\rm (a)} If $0 < p \leq 1$ and $\mathcal{H}$ is a strongly symmetric $p$-normed multilinear hyper-ideal, then
$$\|P\|_{\mathcal{P}^\mathcal{H}}\leq \|P\|_{\mathcal{P}_\mathcal{H}} \leq (n!)^{\frac{1}{p}-1} \|P\|_{\mathcal{P}^\mathcal{H}}$$
for all $n \in \mathbb{N}$ and $P\in {\cal P}_{\cal H}(^nE;F)$.\\
{\rm (b)} If $\mathcal{H}$ is a strongly symmetric normed (Banach) multilinear hyper-ideal, then $\mathcal{P}^\mathcal{H}=\mathcal{P}_\mathcal{H}$ isometrically. In particular, $\mathcal{P}_\mathcal{H}$ is a normed (Banach) polynomial $\ap\frac{n^n}{n!}\fp_{n=1}^\infty$-hyper-ideal. \end{prop}

\begin{proof} (a) Only the second inequality demands a proof. Let $P\in\mathcal{P}_\mathcal{H}(^nE;F)$ and $A\in\mathcal{H}(^nE;F)$ be such that $\widehat{A}=P$. Since $\check{P}=A_s$ and $\mathcal{H}$ is strongly symmetric, we have
\begin{align*}\|P\|_{\mathcal{P}_\mathcal{H}} &= \left(\|\check P\|_{\mathcal{H}}^p\right)^{1/p} =
\left(\|A_s\|_\mathcal{H}^p\right)^{1/p}=\left(\an\dfrac{1}{n!}\cdot \sum\limits_{\sigma\in S_n}A_\sigma\fn_\mathcal{H}^p\right)^{1/p}\\&\le \dfrac{1}{n!}\cdot\left(\sum\limits_{\sigma\in S_n}\|A_\sigma\|_\mathcal{H}^p\right)^{1/p}=\dfrac{1}{n!}\cdot\left(\sum\limits_{\sigma\in S_n}\|A\|_\mathcal{H}^p\right)^{1/p}=\frac{(n!)^{1/p}}{n!}\|A\|_\mathcal{H}.
\end{align*}
Take the infimum over all such multilinear operators $A$ to get the desired inequality.\\
\noindent (b) $\mathcal{P}^\mathcal{H}=\mathcal{P}_\mathcal{H}$ by Lemma \ref{lemasim} and $\|\cdot\|_{\mathcal{P}_\mathcal{H}}=\|\cdot\|_{\mathcal{P}^\mathcal{H}}$ by the item (a) with $p=1$. The result follows from Theorem \ref{hiph}.\end{proof}

It is clear that some extra condition must be imposed on a multilinear hyper-ideal $\cal H$ for ${\cal P}_{\cal H}$ to be a polynomial two-sided ideal. We end this section by providing one such condition, that is related to the so-called {\it factorization method} of generating multilinear ideals from a given operator ideal (cf., e.g., \cite{botelho1}).

\begin{definition}\rm Let $\mathcal{H}$ be a $p$-normed multilinear hyper-ideal and $A\in\mathcal{L}(E_1,\ldots,E_n;F)$. We write that $A\in\mathcal{L}\circ\mathcal{H}(E_1,\ldots,E_n;F)$ if we can find natural numbers $n$ and $r$ with $m r=n$, a Banach space $G$, multilinear operators $B_1\in\mathcal{H}(E_1,\ldots,E_m;G), \ldots
B_m\in\mathcal{H}(E_{(m-1)\cdot r},\ldots,E_n;G)$ and a symmetric multilinear operator $C\in\mathcal{L}_s(^rG;F)$ such that
\begin{equation} A=C\circ(B_1,\ldots,B_r). \label{factmeth}\end{equation} In this case we write $$\|P\|_{\mathcal{L}\circ\mathcal{H}}=\inf\{\|C\|\cdot\|B_1\|_\mathcal{H}
\cdots\|B_r\|_\mathcal{H}\},$$ where the infimum is taken over all factorizations of $A$ as in (\ref{factmeth}).\end{definition}

\begin{prop}{\rm (a)} If $\mathcal{H}$ is a $p$-normed ($p$-Banach) multilinear  hyper-ideal and $\mathcal{L}\circ\mathcal{H}{\longhookrightarrow}\mathcal{H}$,
then $\mathcal{P}^\mathcal{H}$ is a $p$-normed ($p$-Banach) polynomial $\ap\frac{n^n}{n!},\frac{n^n}{n!}\fp_{n=1}^\infty$-two-sided ideal.

\noindent{\rm(b)} If $\mathcal{H}$ is a normed (Banach) strongly symmetric multilinear hyper-ideal and $\mathcal{L}\circ\mathcal{H}{\longhookrightarrow}\mathcal{H}$, then $\mathcal{P}_\mathcal{H}$ is a normed (Banach) polynomial $\ap\frac{n^n}{n!},\frac{n^n}{n!}\fp_{n=1}^\infty$-two-sided ideal.\end{prop}

\begin{proof}(a) Having Theorem \ref{hiph} in mind, we only need to show that if $P\in\mathcal{P}^\mathcal{H}(^nE;F)$ and $R\in\mathcal{P}(^rF;H)$ then $$R\circ P\in\mathcal{P}^\mathcal{H}(^{nr}E;H)\ \mbox{and}\ \|R\circ P\|_{\mathcal{P}^\mathcal{H}}\le\dfrac{r^r}{r!}\|R\|\cdot\|P\|_{\mathcal{P}^\mathcal{H}}^r.$$ For each $A\in\mathcal{H}(^nE;F)$ such that $\hat{A}=P$, it follows easily that
$R\circ P=(\check{R}\circ(A,\ldots,A))^{\wedge}$. Since $$\check{R}\circ(A,\ldots,A)\in\mathcal{L}\circ\mathcal{H}(^{nr}E;H)\subseteq\mathcal{H}(^{nr}E;H),$$ it follows that $R\circ P\in \mathcal{P}^\mathcal{H}(^{nr}E;H)$ and $$\|R\circ P\|_{\mathcal{P}^\mathcal{H}}\le\|\check{R}\circ(A,\ldots,A)\|_\mathcal{H}\le \|\check{R}\circ(A,\ldots,A)\|_{\mathcal{L}\circ\mathcal{H}}\le \|\check{R}\|\cdot\|A\|_\mathcal{H}^r \le\dfrac{r^r}{r!}\|R\|\cdot\|A\|_\mathcal{H}^r.$$ Take the infimum over all such multilinear operators $A$ to get the desired inequality.\\
(b) Having Proposition \ref{corhip}(b) in mind we only need to show that if $P\in\mathcal{P}_\mathcal{H}(^nE;F)$ and $R\in\mathcal{P}(^rF;H)$ then $$R\circ P\in\mathcal{P}_\mathcal{H}(^{nr}E;H)\ \mbox{with} \ \|R\circ P\|_{\mathcal{P}_\mathcal{H}}\le\dfrac{r^r}{r!}\|R\|\cdot\|P\|_{\mathcal{P}_\mathcal{H}}^r.$$
Considering that $\check{R}\circ(\check{P},\ldots,\check{P})\in\mathcal{L}\circ\mathcal{H}(^{nr}E;H)\subseteq\mathcal{H}(^{nr}E;H)$ and that
$\mathcal{H}$ is symmetric, we have $$(R\circ P)^{\vee}=(\check{R}\circ(\check{P},\ldots,\check{P}))_s\in\mathcal{H}(^{nr}E;H).$$ Thus $R\circ P\in\mathcal{P}^\mathcal{H}(^{nr}E;H)$ and
\begin{align*}\|R\circ P\|_{\mathcal{P}_\mathcal{H}}&=\|(R\circ P)^{\vee}\|_{\mathcal{H}}=\|(\check{R}\circ(\check{P},\ldots,\check{P}))_s\|_{\mathcal{H}} \le \dfrac{1}{nr!}\cdot\sum\limits_{\sigma\in S_{nr}}\|(\check{R}\circ(\check{P},\ldots,\check{P}))_\sigma\|_\mathcal{H}\\&= \dfrac{1}{nr!}\cdot\sum\limits_{\sigma\in S_{nr}}\|\check{R}\circ(\check{P},\ldots,\check{P})\|_\mathcal{H}= \|\check{R}\circ(\check{P},\ldots,\check{P})\|_\mathcal{H}\\&\le \|\check{R}\circ(\check{P},\ldots,\check{P})\|_{\mathcal{L}\circ\mathcal{H}}\le\|\check{R}\|\cdot\|\check{P}\|_\mathcal{H}^r\le \dfrac{r^r}{r!}\|R\|\cdot\|P\|_{\mathcal{P}_\mathcal{H}}^r.
\end{align*}\end{proof}

\section{The inequality method}\label{ineqmeth}

In this section and in the next one we provide two methods of generating polynomial hyper/two-sided ideal not passing through multilinear hyper-ideals. We also compare the resulting classes of polynomials with the classes generated by the procedure of Section \ref{multil}. In both cases we give concrete examples of polynomial hyper/two-sided ideals generated by such methods. These methods are polynomial adaptations of the methods introduced in \cite[Sections 2 and 3]{ewerton2}. Let $0 < p,q\leq1$.

\begin{definition}\label{dmdesp}\rm (a) By $\cal BAN$ we denote the class of all Banach spaces over $\mathbb{K}= \mathbb{R}$ or $\mathbb{C}$ and by $p{\rm-}{\cal BAN}$ the class of all $p$-Banach spaces over $\mathbb{K}$. A correspondence $\mathcal{X}\colon \mathcal{BAN} \longrightarrow p{\rm -}\mathcal{BAN}$ that associates to each Banach space $E$  a $p$-Banach space $(\mathcal{X}(E),\|\cdot\|_{\mathcal{X}(E)})$ is called a {\it $p$-sequence functor} if:\\
(i) $\mathcal{X}(E)$ is a linear subspace of $E^{\mathbb{N}}$  with the usual algebraic operations;\\
(ii) For all $x \in E$ and $j \in \mathbb{N}$, we have $(0,\ldots,0,x,0,\ldots)\in \mathcal{X}(E)$, where $x$ is placed at the $j$-th coordinate, and $\|(0,\ldots,0,x,0,\ldots)\|_{\mathcal{X}(E)}=\|x\|_E$.\\
(iii) For every $u\in \mathcal{L}(E;F)$ and every finite sequence $(x_j)_{j=1}^k := (x_1,\ldots,x_k,0,0,\ldots)$ in $E$, $$\|(u(x_j))_{j=1}^k\|_{\mathcal{X}(F)}\le \|u\|\cdot\|(x_j)_{j=1}^k\|_{\mathcal{X}(E)}.$$
When $p=1$ we simply say that $\cal X$ is a \textit{sequence functor}.\\
(b) A $p$-sequence functor $\mathcal{X}$ is \textit{scalarly dominated} by the $q$-sequence functor $\mathcal{Y}$ if, for every finite sequence $(\lambda_j)_{j=1}^k\subseteq\mathbb{K}$,  $\|(\lambda_j)_{j=1}^k\|_{\mathcal{X}(\mathbb{K})}\le \|(\lambda_j)_{j=1}^k\|_{\mathcal{Y}(\mathbb{K})}$.\\
(c) Let $\mathcal{X}$ be a $p$-sequence functor and $\mathcal{Y}$ be a $q$-sequence functor. We say that a $n$-homogenous polynomial $P\in\mathcal{P}(^nE;F)$ is {\it $(\mathcal{X};\mathcal{Y})$-summing} if there is a constant $C>0$ such that, for all finite sequences $(x_j)_{j=1}^k$ in $E$, we have \begin{equation} \label{eq:mdp}\|(P(x_j))_{j=1}^k\|_{\mathcal{Y}(F)}\le C \cdot \sup\limits_{R\in B_{\mathcal{P}(^nE)}}\|(R(x_j))_{j=1}^k\|_{\mathcal{X}(\mathbb{K})}.\end{equation} In this case we write $P\in\mathcal{P(X;Y)}(^nE;F)$ and define $$\|P\|_{\mathcal{P(X;Y)}}=\inf\{C>0 : C\ \mbox{satisfies} \ (\ref{eq:mdp})\}.$$\end{definition}

For a number of examples of sequence functors, see \cite[Example 2.2]{ewerton2}.

\begin{teo}\label{teometdes} Let $\mathcal{X}$ be a $p$-sequence functor and $\mathcal{Y}$ be a $q$-sequence functor with $\mathcal{Y}$ scalarly dominated by $\mathcal{X}$. Then the class $\mathcal{P(X;Y)}$ is a $q$-Banach polynomial hyper-ideal.\end{teo}
\begin{proof} We shall just check the hyper-ideal property (the remaining conditions follow the same steps of the proof of \cite[Theorem 2.6]{ewerton2}).
Let $P\in\mathcal{P(X;Y)}(^nE;F)$, $Q\in\mathcal{P}(^mG;F)$ and $t\in\mathcal{L}(F;H)$. Of course we can assume $Q\neq0$. For every finite sequence $(x_j)_{j=1}^k$ in $G$, \begin{align*}\|(t\circ P\circ Q(x_j))&_{j=1}^k\|_{\mathcal{Y}(H)}=\|(t(P(Q(x_j))))_{j=1}^k\|_{\mathcal{Y}(H)}\le \|t\|\cdot\|((P(Q(x_j))))_{j=1}^k\|_{\mathcal{Y}(F)}\\&\le \|t\|\cdot\|P\|_{\mathcal{P(X;Y)}}\cdot\sup\limits_{R\in B_{\mathcal{P}(^nE)}}\|(R(Q(x_j)))_{j=1}^k\|_{\mathcal{X}(\mathbb{K})}\\&=
\|t\|\cdot\|P\|_{\mathcal{P(X;Y)}}\cdot \|Q\|^n\cdot\sup\limits_{R\in B_{\mathcal{P}(^nE)}}\an\ap\ap \dfrac{R(Q(x_j))}{\|Q\|^n}\fp\fp_{j=1}^k\fn_{\mathcal{X}(\mathbb{K})}\\&
\leq \|t\|\cdot\|P\|_{\mathcal{P(X;Y)}}\cdot\|Q\|^n\cdot\sup\limits_{S\in B_{\mathcal{P}(^{mn}G)}}\|((S(x_j)))_{j=1}^k\|_{\mathcal{X}(\mathbb{K})}.
\end{align*}
Thus $t\circ P\circ Q\in\mathcal{P(X;Y)}(^{mn}G;H)$ and $\|t\circ P\circ Q\|_{\mathcal{P(X;Y)}}\le\|t\|\cdot\|P\|_{\mathcal{P(X;Y)}}\cdot\|Q\|^n$.\end{proof}

Let $\mathcal{(X;Y)}$ be the multilinear hyper-ideal determined by the sequence functors $\cal X$ and $\cal Y$ (cf. \cite[Definition 2.4]{ewerton2}).
Our next goal is to compare the class $\mathcal{P(X;Y)}$ with the classes $\mathcal{P}_{\mathcal{(X;Y)}}$ and $\mathcal{P}^{\mathcal{(X;Y)}}$.

\begin{lema}\label{lemafortsimet} Let $\mathcal{X}$ be a $p$-sequence functor scalarly dominated by the $q$-sequence functor $\mathcal{Y}$. Then the $q$-Banach hyper-ideal $(\mathcal{X};\mathcal{Y})$ is strongly symmetric.\end{lema}

\begin{proof} Let $A \in (\mathcal{X};\mathcal{Y})(^nE;F)$,
$\sigma\in S_n$ and $(x_j^1)_{j=1}^k,\ldots,(x_j^n)_{j=1}^k$ be finite sequences in $E$. Using that $\|T\| = \|T_{\sigma}\|$ for every  $T\in {\mathcal{L}(^nE;F)}$, we have
\begin{eqnarray*}\|(A_\sigma(x_j^1,\ldots,x_j^n))_{j=1}^k\|_{\mathcal{Y}(F)}&\leq&
\|A\|_{(\mathcal{X};\mathcal{Y})}\cdot\sup\limits_{T\in B_{\mathcal{L}(^nE;F)}} \|(T(x_j^{\sigma(1)},\ldots,x_j^{\sigma(n)}))_{j=1}^k\|_{\mathcal{X}(\mathbb{K})}\\&=& \|A\|_{(\mathcal{X};\mathcal{Y})}\cdot\sup\limits_{T\in B_{\mathcal{L}(^nE;F)}}\|(T(x_j^{1},\ldots,x_j^n))_{j=1}^k\|_{\mathcal{X}(\mathbb{K})}.\end{eqnarray*}
Thus $A_\sigma\in(\mathcal{X};\mathcal{Y})(^nE;F)$ and $\|A_\sigma\|_{(\mathcal{X};\mathcal{Y})}\le\|A\|_{(\mathcal{X};\mathcal{Y})}$.
Since $A=A_{\sigma\circ\sigma^{-1}}=(A_\sigma)_{\sigma^{-1}}$, from what we have just proved it follows that  $$\|A\|_{(\mathcal{X};\mathcal{Y})}=\|(A_\sigma)_{\sigma^{-1}}\|_{(\mathcal{X};\mathcal{Y})}\le\|A_\sigma\|_{(\mathcal{X};\mathcal{Y})}.$$ Therefore $\|A\|_{(\mathcal{X};\mathcal{Y})}=\|A_\sigma\|_{(\mathcal{X};\mathcal{Y})}$.\end{proof}

\begin{prop}\label{propimp}Let $\mathcal{X}$ be a $p$-sequence functor and $\mathcal{Y}$ be a $q$-sequence functor scalarly dominated by $\mathcal{X}$. Then
$\mathcal{P}_{\mathcal{(X;Y)}} = \mathcal{P}^{\mathcal{(X;Y)}}\subseteq\mathcal{P(X;Y)}$ and $$\|\cdot\|_{\mathcal{P(X;Y)}}\le\|\cdot\|_{\mathcal{P}^{\mathcal{(X;Y)}}}\le\|\cdot\|_{\mathcal{P}_{\mathcal{(X;Y)}}}.$$\end{prop}
\begin{proof} The equality $\mathcal{P}_{\mathcal{(X;Y)}} = \mathcal{P}^{\mathcal{(X;Y)}}$ follows from Lemma \ref{lemafortsimet} and Proposition \ref{corhip}. The inequality $\|\cdot\|_{\mathcal{P^{(X;Y)}}}\le\|\cdot\|_{\mathcal{P}_{\mathcal{(X;Y)}}}$ is obvious. Given $P\in\mathcal{P}^{\mathcal{(X;Y)}}(^nE;F)$, there exists $A\in\mathcal{(X;Y)}(^nE;F)$ such that $\widehat{A} = P$. For $x_1, \ldots, x_k \in E$, since $\widehat{T}\in B_{\mathcal{P}(^nE)}$ for every $T\in B_{\mathcal{L}(^nE)}$, we have
\begin{eqnarray*}\|(P(x_j))_{j=1}^k\|_{\mathcal{Y}(F)}&=&\|(A(x_j,\ldots,x_j))_{j=1}^k\|_{\mathcal{Y}(F)}
\\&\le& \|A\|_{\mathcal{(X;Y)}}\cdot\sup\limits_{T\in B_{\mathcal{L}(^nE)}} \|(T(x_j,\ldots,x_j))_{j=1}^k\|_{\mathcal{X}(\mathbb{K})}\\&=&\|A\|_{\mathcal{(X;Y)}}\cdot\sup\limits_{T\in B_{\mathcal{L}(^nE)}} \|(\widehat{T}(x_j))_{j=1}^k\|_{\mathcal{X}(\mathbb{K})} \\&\le& \|A\|_{\mathcal{(X;Y)}}\cdot\sup\limits_{R\in B_{\mathcal{P}(^nE)}} \|(R(x_j))_{j=1}^k\|_{\mathcal{X}(\mathbb{K})},
\end{eqnarray*}
Therefore $P\in\mathcal{P(X;Y)}$ and $\|P\|_{\mathcal{P(X;Y)}}\le \|A\|_{\mathcal{(X;Y)}}$. Taking the infimum over all multilinear operators $A$ in $\mathcal{(X;Y)}$ such that $\widehat{A} = P$ we get $\|P\|_{\mathcal{P(X;Y)}}\le\|P\|_{\mathcal{P}^{\mathcal{(X;Y)}}}.$\end{proof}

For $p >0$, by $\ell_p(\cdot)$ we denote the $p$-sequence functor $E \mapsto \ell_p(E)$.

\begin{ex}\rm Let us see that the inclusion in the proposition above is strict. Consider the sequences functors  $\mathcal{X}=\mathcal{Y}=\ell_1(\cdot)$ and the 2-homogenous polynomial $$P\colon \ell_2\longrightarrow\ell_2\widehat{\otimes}_\pi\ell_2~,~P(x)=x\otimes x.$$ Noticing that $\check{P}(x,y)=\frac{x\otimes y + y\otimes x}{2}$, what was proved in \cite[Example 3.4]{dimant} means, in our notation, that $P\in\mathcal{P}{(\ell_1(\cdot);\ell_1(\cdot))}(^2\ell_2;\ell_2\widehat{\otimes}_\pi\ell_2)$ but $\check{P}\notin(\ell_1(\cdot);\ell_1(\cdot))(^2\ell_2;\ell_2\widehat{\otimes}_\pi\ell_2)$. Therefore $ \mathcal{P}_{{(\ell_1(\cdot);\ell_1(\cdot))}} \varsubsetneqq\mathcal{P}(\ell_1(\cdot);\ell_1(\cdot))$.\end{ex}

Next we show that the inequality method recovers, as a particular instance, the following important class of polynomials introduced by V. Dimant \cite[Definition 3.1]{dimant}:

\begin{definition}\rm A polynomial $P\in\mathcal{P}(^nE;F)$ is said to be {\it strongly $p$-summing}, $p > 0$, if there is a constant $C>0$ such that, for every finite sequence $(x_j)_{j=1}^k$ in $E$, \begin{equation} \label{eq:sspp}\ap\sum\limits_{j=1}^k\|P(x_j)\|^p\fp^{1/p}\le C\cdot \sup\limits_{q\in B_{\mathcal{P}(^nE)}}\ap\sum\limits_{j=1}^k|q(x_j)|^p\fp^{1/p}.\end{equation} In this case we write $P\in\mathcal{P}_{ss}^p(^nE;F)$ and define $\|P\|_{\mathcal{P}_{ss}^p}=\inf\{C :  C\ \mbox{satisfies} \ (\ref{eq:sspp})\}.$\end{definition}

\begin{prop} $\mathcal{P}_{ss}^p$ is a $p$-Banach polynomial hyper-ideal if $0<p<1$ and a Banach polynomial hyper-ideal if $p\ge1$.\end{prop}
\begin{proof}Apply Theorem \ref{teometdes} with ${\cal X}(E) = {\cal Y}(E) = \ell_p(E)$.\end{proof}

It is a natural question whether the inequality method, under stronger assumptions, generate polynomial two-sided ideals. Let us see that, unfortunately, the answer to this question is unsatisfactory. Given $P\in\mathcal{P}(\mathcal{X};\mathcal{Y})(^nE;F)$ and $R\in\mathcal{P}(^rF;H)$, in order to prove, under the assumptions of Theorem \ref{teometdes}, that $R \circ P$ belongs to $\mathcal{P}{\mathcal{(X;Y)}}$, the following two conditions must be satisfied:

\medskip

\noindent (C1) If $Q \in\mathcal{P}(^nE;F)$, $k\in\mathbb{N}$ and $x_1, \ldots, x_k \in E$, then $$\|(Q(x_j))_{j=1}^k\|_{\mathcal{Y}(F)}\le \|Q\|\cdot\|(x_j)_{j=1}^k\|_{\mathcal{Y}(E)}^n.$$
(C2) If $k\in\mathbb{N}$ and $\lambda_1, \ldots, \lambda_k \in \mathbb{K}$, then
\begin{equation*}\|(\lambda_j)_{j=1}^k\|_{\mathcal{X}(\mathbb{K})}^r \leq \|(\lambda_j^r)_{j=1}^k\|_{\mathcal{X}(\mathbb{K})}.\label{eq:desind}\end{equation*}

It is true that, supposing (C1) and (C2), $\mathcal{P}{\mathcal{(X;Y)}}$ is a polynomial Banach two-sided ideal. We do not formalize this because the examples of sequence functors satisfying condition (C2) we are aware of lead to a trivial situation. In fact, the only usual sequence functor norm satisfying condition (C2) is the supremum norm. But this case is highly uninteresting:

\begin{prop}Let $\mathcal{X}$ and $\mathcal{Y}$ be sequence functors with $\mathcal{Y}$ scalarly dominated by $\mathcal{X}$. If  $\|(x_j)_{j=1}^k\|_{\mathcal{Y}(E)}=\sup\limits_{j=1,\ldots,k}\|x_j\|$ for all $E,k$ and $x_1, \ldots, x_k \in E$, then $\mathcal{P}(\mathcal{X};\mathcal{Y})(^nE;F)=\mathcal{P}(^nE;F)$ isometrically for all $E,F$ and $n$.\end{prop}
\begin{proof}Let $P\in\mathcal{P}(^nE;F)$. By the Hahn-Banach Theorem it follows that
$$\|P(x)\|\le\|P\|\cdot\sup\limits_{q\in B_{\mathcal{P}(^nE)}}|q(x)|$$
for every $x \in E$. Then, given $x_1, \ldots, x_k \in E$, \begin{align*}\|(P(x_j))_{j=1}^k\|_{\mathcal{Y}(F)}&=\sup\limits_{j=1,\ldots,k}\|P(x_j)\|\le\|P\|\cdot\sup\limits_{q\in B_{\mathcal{P}(^nE)}}\ap\sup\limits_{j=1,\ldots,k}|q(x_j)|\fp\\&=\|P\|\cdot\sup\limits_{q\in B_{\mathcal{P}(^nE)}}\|(q(x_j))_{j=1}^k\|_{\mathcal{Y}(\mathbb{K})}\le\|P\|\cdot\sup\limits_{q\in B_{\mathcal{P}(^nE)}}\|(q(x_j))_{j=1}^k\|_{\mathcal{X}(\mathbb{K})},\end{align*}  proving that $P\in\mathcal{P}(\mathcal{X};\mathcal{Y})(^nE;F)$ and $\|P\|_{\mathcal{P}(\mathcal{X};\mathcal{Y})}=\|P\|$.\end{proof}

One can argue that the we should have the supremum norm only on scalar-valued sequences. This is correct, but the search for sequence functors having the supremum norm on scalar-valued sequences and some different norm on vector-valued sequences would lead to artificial constructions, in which we are not interested.

\section{The boundedness method}\label{boundmeth}

In this section we show that the polynomial counterpart of the boundedness method for multilinear hyper-ideals \cite[Section 3]{ewerton2} generates polynomial hyper-ideals and that a (reasonable) variant of it generates polynomial two-sided ideals. $\mathcal{I}$-bounded polynomials were first introduced by Aron and Rueda \cite{aronrueda2}. Let $0 < p \leq 1$.

\begin{definition}\label{ilimpoldef}\rm Let $\mathcal{I}$ be a $p$-normed operator ideal. A polynomial $P\in\mathcal{P}(^nE;F)$ is {\it $\mathcal{I}$-bounded} if $P(B_E) \in {\cal C}_{\cal I}(F)$, that is, if there are a Banach space $H$ and a linear operator $u\in\mathcal{I}(H;F)$ such that $P(B_E)\subseteq u(B_H)$. In this case we write $P\in\mathcal{P}_{\langle\mathcal{I}\rangle}(^nE;F)$ and define $$\|P\|_{\mathcal{P}_{\langle\mathcal{I}\rangle}}=\inf\{\|u\|_\mathcal{I}: P(B_E)\subseteq u(B_H)\}.$$\end{definition}

Reasoning as in the proof of \cite[Theorem 3.3]{ewerton2} with the help of the series criterion for polynomial hyper-ideals (Theorem \ref{cs}), we get the

\begin{teo}\label{teoilimpol}Let $\mathcal{I}$ be a $p$-Banach operator ideal. Then $\mathcal{P}_{\langle\mathcal{I}\rangle}$ is a $p$-Banach polynomial hyper-ideal.\end{teo}

\begin{ex}\label{exexex}\rm Let $\mathcal{K}$ and $\mathcal{W}$ denote, respectively, the closed ideals of compact and weakly compact operator. Reasoning as in \cite[Example 3.1]{aronrueda2}, we see that $\mathcal{C}_\mathcal{K}(E)$ is the collection of relatively compact subsets of $E$ and that $\mathcal{C}_\mathcal{W}(E)$ is the collection of relatively weakly compact subsets of $E$. Theorem \ref{teoilimpol} gives a second proof of the fact that the classes $\mathcal{P}_\mathcal{K}$ of compact homogenous polynomials and $\mathcal{L}_\mathcal{W}$ of weakly compact homogenous polynomials are closed  polynomial hyper-ideals.\end{ex}

By $\mathcal{L}_{\langle\mathcal{I}\rangle}$ we denote the multilinear hyper-ideal generated by the boundedness method (see \cite[Section 3]{ewerton2}), that is, a multilinear operator $A \in {\cal L}(E_1, \ldots, E_n;F)$ belongs to $\mathcal{L}_{\langle\mathcal{I}\rangle}$ if $A(B_{E_1}\times \cdots \times B_{E_n}) \in {\cal C}_{\cal I}(F)$. Next we compare the polynomial hyper-ideal $\mathcal{P}_{\langle\mathcal{I}\rangle}$ with the classes obtained from $\mathcal{L}_{\langle\mathcal{I}\rangle}$ by the method of Section \ref{multil}.

\begin{prop}Let $0 < p \leq 1$ and $\mathcal{I}$ be a $p$-Banach operator ideal. Then $\mathcal{P}_{\langle\mathcal{I}\rangle}=\mathcal{P}_{\mathcal{L}_{\langle\mathcal{I}\rangle}}=\mathcal{P}^{\mathcal{L}_{\langle\mathcal{I}\rangle}}$ and  $$\|P\|_{\mathcal{P}^{\mathcal{L}_{\langle\mathcal{I}\rangle}}}\le\|P\|_{\mathcal{P}_{\mathcal{L}_{\langle\mathcal{I}\rangle}}}\le (n!)^{\frac{1}{p}-1}\|P\|_{\mathcal{P}^{\mathcal{L}_{\langle\mathcal{I}\rangle}}} {\rm ~and~} \|P\|_{\mathcal{P}_{\langle\mathcal{I}\rangle}}\le\|P\|_{\mathcal{P}_{\mathcal{L}_{\langle\mathcal{I}\rangle}}}\le \dfrac{n^n}{n!}\|P\|_{\mathcal{P}_{\langle\mathcal{I}\rangle}},$$ for every $P\in\mathcal{P}_{\langle\mathcal{I}\rangle}(^nE;F)$.\end{prop}

\begin{proof} Using that $A_\sigma(B_E\times\cdots\times B_E)= A(B_E\times\cdots\times B_E)$ for every $n$-linear operator $A$ and every permutation $\sigma \in S_n$, it is not difficult to check that the multilinear hyper-ideal $\mathcal{L}_{\langle\mathcal{I}\rangle}$ is strongly symmetric. Proposition \ref{corhip} gives $\mathcal{P}_{\mathcal{L}_{\langle\mathcal{I}\rangle}}=\mathcal{P}^{\mathcal{L}_{\langle\mathcal{I}\rangle}}$ and $$\|\cdot\|_{\mathcal{P}^{\mathcal{L}_{\langle\mathcal{I}\rangle}}}\le
\|\cdot\|_{\mathcal{P}_{\mathcal{L}_{\langle\mathcal{I}\rangle}}}\le (n!)^{\frac{1}{p}-1}\|\cdot\|_{\mathcal{P}^{\mathcal{L}_{\langle\mathcal{I}\rangle}}},$$ on $\mathcal{P}_{\mathcal{L}_{\langle\mathcal{I}\rangle}}(^nE;F)$. Given
$P\in\mathcal{P}_{{\langle\mathcal{I}\rangle}}(^nE;F)$, then exist a Banach space $H$ and a linear operator $u\in\mathcal{I}(H;F)$ such that $P(B_E)\subseteq u(B_H)$. Let $x_1,\ldots,x_n\in B_E$ and $\varepsilon_1, \ldots, \varepsilon_n=\pm1$. Since $\displaystyle\frac{\varepsilon_1x_1 + \cdots + \varepsilon_nx_n}{n} \in B_E$, there is $z_{\varepsilon_1, \ldots, \varepsilon_n} \in B_H$ such that $$P\left(\displaystyle\frac{\varepsilon_1x_1 + \cdots + \varepsilon_nx_n}{n}\right) = u\left(z_{\varepsilon_1, \ldots, \varepsilon_n}\right).$$
Then $w:=\dfrac{1}{2^n}\sum\limits_{\varepsilon_j=\pm1}\varepsilon_1\cdots\varepsilon_nz_{\varepsilon_1,\ldots,\varepsilon_n} \in B_H$. By the polarization formula, \begin{eqnarray*}\check{P}(x_1,\ldots,x_n)&=&\dfrac{1}{n!2^n}\sum\limits_{\varepsilon_j=\pm1}\varepsilon_1\cdots\varepsilon_nP(\varepsilon_1x_1+\cdots+\varepsilon_nx_n) \\&=& \dfrac{1}{n!2^n}\sum\limits_{\varepsilon_j=\pm1}\varepsilon_1\cdots\varepsilon_nn^nP\ap\dfrac{\varepsilon_1x_1+\cdots+\varepsilon_nx_n}{n}\fp\\&=&
\dfrac{n^n}{n!2^n}\sum\limits_{\varepsilon_j=\pm1}\varepsilon_1\cdots\varepsilon_nu(z_{\varepsilon_1,\ldots,\varepsilon_n}) =\dfrac{n^n}{n!}u(w) \in \frac{n^n}{n!}u(B_H).\end{eqnarray*}
Since $\frac{n^n}{n!}u \in {\cal I}(H;F)$, we conclude that $\check{P}(B_E\times\cdots\times B_E)\subseteq \frac{n^n}{n!}u(B_H),$ proving that $\check{P}\in\mathcal{L}_{\langle\mathcal{I}\rangle}(^nE;F)$ and $$\|P\|_{{\cal P}_{\mathcal{L}_{\langle\mathcal{I}\rangle}}}= \|\check{P}\|_{\mathcal{L}_{\langle\mathcal{I}\rangle}}\le\an\dfrac{n^n}{n!}u\fn_\mathcal{I}=\dfrac{n^n}{n!}\cdot\|u\|_\mathcal{I}.$$ Then $\|P\|_{\mathcal{P}_{\mathcal{L}_{\langle\mathcal{I}\rangle}}}\le \frac{n^n}{n!}\cdot\|P\|_{\mathcal{P}_{{\langle\mathcal{I}\rangle}}}$.

Conversely, the inclusion $\mathcal{P}_{\mathcal{L}_{\langle\mathcal{I}\rangle}}\subseteq\mathcal{P}_{{\langle\mathcal{I}\rangle}}$ and the inequality $\|\cdot\|_{\mathcal{P}_{{\langle\mathcal{I}\rangle}}}\le\|\cdot\|_{\mathcal{P}_{\mathcal{L}_{\langle\mathcal{I}\rangle}}}$ follow easily from  the definitions.\end{proof}

The linear operator $u$ in Definition \ref{ilimpoldef} makes clear that $\mathcal{P}_{{\langle\mathcal{I}\rangle}}$ is not expected to be a two-sided ideal. Our aim is to introduce a variant of the boundedness method that generates polynomial two-sided ideals. Curiously enough, we found our way looking at another method of generating polynomial ideals that does not work well in our nonlinear setting, namely, the factorization method, which goes back to Pietsch in \cite{pietsch}:

\begin{definition}\label{metfatpol}\rm
Given a $p$-normed operator ideal $\mathcal{I}$, $0 < p \leq 1$, and an $n$-homogenous polynomial $P \in \mathcal{P}(^nE;F)$, we write $P\in{\cal P} \circ {\cal I}(^nE;F)$ if there exist a Banach space $G$, a linear operator $u\in\mathcal{I}(E;G)$ and an $n$-homogenous polynomial $Q \in \mathcal{P}(^nG;F)$ such that $P=Q\circ u$, and we define $$\|P\|_{{\cal P} \circ {\cal I}}=\inf\{\|Q\|\cdot\|u\|_{\mathcal{I}}^n : P=Q\circ u {\rm ~with~} u \in {\cal I}\}.$$\end{definition}

It is well known that ${\cal P} \circ {\cal I}$ is a $(\frac{p}{n})_{n=1}^\infty$-normed polynomial ideal, meaning that $\|\cdot\|_{{\cal P} \circ {\cal I}}$ is a $\frac{p}{n}$-norm on each component ${\cal P} \circ {\cal I}(^nE;F)$ (see, e.g., \cite{botelho1}).

As can be seen in \cite[Example 1.1]{ewerton2}, the factorization method does not generate even polynomial hyper-ideals in general. Next we show that a combination of the boundedness method (that generates polynomial hyper-ideals) with the factorization method (that generates polynomial ideals) is effective in the generation of two-sided-ideals.

\begin{definition}\label{defoo}\rm Let $\mathcal{I}$ be a $p$-normed operator ideal. For a polynomial $P\in \mathcal{P}(^nE;F)$, if we can find a Banach space $H$ and a polynomial $R\in {\cal P} \circ {\cal I}(^nH;F)$  such that \begin{equation}\label{eq:milimp}P(B_E)\subseteq R(B_H),\end{equation} we write $P\in \mathcal{P}_{\mathcal{Q}(\mathcal{I})}(^nE;F)$ and define $\|\cdot\|_{\mathcal{P}_{\mathcal{Q}(\mathcal{I})}}\colon \mathcal{P}_{\mathcal{Q}(\mathcal{I})}\longrightarrow [0,\infty)$ by $$\|P\|_{\mathcal{P}_{\mathcal{Q}(\mathcal{I})}}=\inf\{\|R\|_{{\cal P} \circ {\cal I}}:  R  \ \mbox{satisfies} \ \eqref{eq:milimp}\}.$$\end{definition}

\begin{teo}Let $0 < p \le 1$ and $\mathcal{I}$ be a $p$-normed ($p$-Banach) operator ideal. Then $\mathcal{P}_{\mathcal{Q}(\mathcal{I})}$ is a $(\frac{p}{n})_{n=1}^\infty$-normed ($(\frac{p}{n})_{n=1}^\infty$-Banach) polynomial two-sided ideal.\end{teo}
\begin{proof} Let us check the two-sided ideal property. The other conditions of Definition \ref{dhip} (or of Theorem \ref{cs} in the complete case) are easily verified having in mind that ${\cal P} \circ {\cal I}$ is a $(\frac{p}{n})_{n=1}^\infty$-normed (Banach) polynomial ideal.
Let $P\in\mathcal{P}_{\mathcal{Q}(\mathcal{I})}(^nE;F)$, $Q\in\mathcal{P}(^mG;E)$, $R\in\mathcal{P}(^rF;H)$ and $\varepsilon>0$. On the one hand, we can find a Banach space $E_1$ and a polynomial $P_1\in {\cal P} \circ {\cal I}(^nE_1;F)$ such that $P(B_E)\subseteq P_1(B_{E_1})$ and
\begin{equation}\label{fghg}
\|P_1\|_{{\cal P} \circ {\cal I}}<(1+\varepsilon)\|P\|_{\mathcal{P}_{\mathcal{Q}(\mathcal{I})}}.
\end{equation} Then $P\circ Q(B_G)\subseteq \|Q\|^nP_1(B_{E_1}).$ Since $\|Q\|^nP_1\in {\cal P} \circ {\cal I}(^nE_1;F)$, we conclude that $P\circ Q\in\mathcal{P}_{\mathcal{Q}(\mathcal{I})}(^{nm}G;F)$ and $$\|P\circ Q\|_{\mathcal{P}_{\mathcal{Q}(\mathcal{I})}}\le\|(\|Q\|^nP_1)\|_{{\cal P} \circ {\cal I}} <(1+\varepsilon)\|P\|_{\mathcal{P}_{\mathcal{Q}(\mathcal{I})}}\cdot\|Q\|^n.$$
Making $\varepsilon \longrightarrow 0^+$, we get $\|P\circ Q\|_{\mathcal{P}_{\mathcal{Q}(\mathcal{I})}}\le
\|P\|_{\mathcal{P}_{\mathcal{Q}(\mathcal{I})}}\cdot\|Q\|^n$.

On the other hand, from (\ref{fghg}) there are a Banach space $F_1$, an operator $u\in\mathcal{I}(E_1;F_1)$ and a  polynomial $S\in\mathcal{P}(^nF_1;F)$ such that $P_1=S\circ u$ and $$ \|S\|\cdot\|u\|_{\mathcal{I}}^n  <(1+\varepsilon)\|P\|_{\mathcal{P}_{\mathcal{Q}(\mathcal{I})}}.$$
Then $R\circ P(B_E)\subseteq (R\circ S)\circ u(B_{E_1}).$
Since $(R\circ S)\circ u \in {\cal P} \circ {\cal I}(^{nr}E_1;H)$, we conclude that $R\circ P\in\mathcal{P}_{\mathcal{Q}(\mathcal{I})}(^{nr}E;H)$ and $$\|R\circ P\|_{\mathcal{P}_{\mathcal{Q}(\mathcal{I})}}\le\|R\circ S\|\cdot\|u\|_{\mathcal{I}}^{rn}\le \|R\|(\|S\|\cdot\|u\|_{\mathcal{I}}^{n})^r\le(1+\varepsilon)^r\|R\|\cdot\|P\|_{\mathcal{P}_{\mathcal{Q}(\mathcal{I})}}^r.$$ Making $\varepsilon \longrightarrow 0^+$, we get $\|R\circ P\|_{\mathcal{P}_{\mathcal{Q}(\mathcal{I})}}\le\|R\|\cdot\|P\|_{\mathcal{P}_{\mathcal{Q}(\mathcal{I})}}^r$. Thus, the two-sided ideal property follows.\end{proof}

\section{Composition polynomial ideals}\label{compideals}

In this section we show that the classical composition polynomial ideal ${\cal I}\circ \cal {\cal P}$, that goes back to Pietsch \cite{pietsch}, is a polynomial hyper-ideal for every operator ideal $\cal I$, and we establish a sufficient (and necessary) condition on $\cal I $ for  ${\cal I}\circ \cal {\cal P}$ to be a polynomial two-sided-ideal. Examples of operator ideals satisfying such condition are provided.

\begin{definition}\label{pidcom}\rm Let $\mathcal{I}$ is a $p$-normed operator ideal. A polynomial $P\in \mathcal{P}(^nE;F)$ is said to belong to $\mathcal{I}\circ \mathcal{P}$ if there are a Banach space $G$, a polynomial $Q\in\mathcal{P}(^nE;G)$ and an operator $u\in\mathcal{I}(G;F)$ such that $P=u\circ Q.$ Define $\|\cdot\|_{\mathcal{I}\circ\mathcal{P}}\colon\mathcal{I}\circ\mathcal{P}\longrightarrow[0,\infty)$ by $$\|P\|_{\mathcal{I}\circ\mathcal{P}}=\inf\{\|u\|_\mathcal{I}\cdot\|Q\|: P=u\circ Q, u \in {\cal I}\}.$$\end{definition}

More information on composition ideals can be found in \cite{bpr2}. Actually, there are two natural norms on $\mathcal{I}\circ \mathcal{P}$ (cf. \cite[Definition 3.6(b)]{bpr2}). The results we prove in this section will make clear that the norm we have chosen is, in fact, more appropriate to the hyper/two-sided ideals environment.

\begin{teo}\label{thipic}If $\mathcal{I}$ is a $p$-normed ($p$-Banach) operator ideal, then $\mathcal{I}\circ\mathcal{P}$ is a $p$-normed ($p$-Banach) polynomial hyper-ideal.\end{teo}

\begin{proof} It is well known that $\mathcal{I}\circ\mathcal{P}$ is a (complete if $\mathcal{I}$ is complete) polynomial ideal, so we just have to check the hyper-ideal property. Given  $P\in\mathcal{I}\circ\mathcal{P}(^nE;F)$, $R\in\mathcal{P}(^mH;E)$ and $t\in\mathcal{L}(F;H_1)$, write $P=u\circ Q$ where $G$ is a Banach space, $u\in\mathcal{I}(G;F)$ and $Q\in\mathcal{P}(^nE;G)$. Since $$t\circ P\circ R = t\circ (u\circ Q)\circ R=(t\circ u)\circ(Q\circ R),$$ $t\circ u\in \mathcal{I}(G;H_1)$ and $Q\circ R\in\mathcal{P}(^{mn}H;G)$, we have $t\circ P \circ R \in \mathcal{I}\circ\mathcal{P}(^{mn}H;H_1).$
Moreover, $$\|t\circ P\circ R\|_{\mathcal{I}\circ\mathcal{P}} = \|(t\circ u)\circ(Q\circ R)\|_{\mathcal{I}\circ\mathcal{P}} \leq \|t\circ u\|_\mathcal{I}\cdot\|Q\circ R\|\le \|t\|\cdot(\|u\|_\mathcal{I}\cdot\|Q\|)\cdot\|R\|^n.$$ The hyper-ideal inequality follows by taking the infimum over all such factorizations of $P$.\end{proof}

\begin{ex}\label{excfc}\rm Ryan \cite{ryanthesis} proved that $${\cal P}_{\cal K} =  \mathcal{K}\circ\mathcal{P} {\rm ~~ and~~ } {\cal P}_{\cal W} =  \mathcal{W}\circ\mathcal{P}.$$
So, Theorem \ref{thipic} gives one more proof that the classes ${\cal P}_{\cal K}$ of compact polynomials and ${\cal P}_{\cal W}$ of weakly compact polynomials are closed polynomial hyper-ideals.
\end{ex}

To give another remarkable example we need of the following definition introduced by Aron and Rueda \cite{aronrueda3} (the linear case is due to Sinha and Karn \cite{sinha}).

\begin{definition}\rm Let $p,q \geq 1$ be such that $\frac{1}{p}+\frac{1}{q}=1$. An $n$-homogenous polynomial $P\in\mathcal{P}(^nE;F)$ is called {\it $p$-compact} if there is a sequence $(x_j)_{j=1}^\infty\in\ell_p(F)$ such that \begin{equation}\label{eq:pcomp}P(B_E)\subseteq\left\{\sum\limits_{j=1}^\infty
\lambda_jx_j :  (\lambda_j)_{j=1}^\infty\in B_{\ell_q} \right\}.\end{equation}
In this case we write $P\in\mathcal{P}_{\mathcal{K}_p}(^nE;F)$ and define  $$\|P\|_{\mathcal{P}_{\mathcal{K}_p}}=\inf\{\|(x_j)_{j=1}^\infty\|_p :  (x_j)_{j=1}^\infty \mbox{~satisfies} \ (\ref{eq:pcomp})\}.$$
Making $n=1$ we recover the extensively studied Banach ideal $\mathcal{K}_p$ of $p$-compact linear operators (see, e.g., \cite{pietschpams} and references therein).\end{definition}

\begin{ex}\rm Combining \cite[Theorem 3.1]{aronrueda3} and \cite[Proposition 3.2]{bpr2} we get that $\mathcal{P}_{\mathcal{K}_p}={\mathcal{K}_p}\circ\mathcal{P}$ isometrically. Therefore the class $\mathcal{P}_{\mathcal{K}_p}$ of $p$-compact polynomials is a Banach polynomial hyper-ideal by Theorem \ref{thipic}.
\end{ex}

By $\mathcal{I}\circ\mathcal{L}$ we denote the composition multilinear ideal, that is, $A \in \mathcal{I}\circ\mathcal{L}(E_1, \ldots, E_n;F)$ if there are a Banach space $G$, an $n$-linear operator $B \in {\cal L}(E_1, \ldots, E_n;G)$ and an operator $u \in {\cal I}(G;F)$ such that $A = u \circ B$. The function $\|\cdot\|_{\mathcal{I}\circ\mathcal{L}}$ is defined in the obvious way.

\begin{prop}For every $p$-normed ($p$-Banach) operator ideal $\cal I$, $$\mathcal{I}\circ\mathcal{P}=\mathcal{P}_{\mathcal{I}\circ\mathcal{L}}=\mathcal{P}^{\mathcal{I}\circ\mathcal{L}}.$$
And if $P\in\mathcal{I}\circ\mathcal{P}(^nE;F)$, then
$$\|P\|_{\mathcal{P}^{\mathcal{I}\circ\mathcal{L}}}\le\|P\|_{\mathcal{P}_{\mathcal{I}\circ\mathcal{L}}}\le (n!)^{\frac{1}{p}-1}\|P\|_{\mathcal{P}^{\mathcal{I}\circ\mathcal{L}}} ~~and~~ \|P\|_{\mathcal{I}\circ\mathcal{P}}\le\|P\|_{\mathcal{P}_{\mathcal{I}\circ\mathcal{L}}}\le c(n,E)\|P\|_{\mathcal{I}\circ\mathcal{P}},$$ where $c(n,E)$ is the $n$-th polarization constant of the Banach space $E$.\end{prop}

\begin{proof} The class $\mathcal{I}\circ\mathcal{L}$ of multilinear operators is strongly symmetric (cf. \cite[Proposition 3.3]{bpr2}), so $\mathcal{P}_{\mathcal{I}\circ\mathcal{L}}=\mathcal{P}^{\mathcal{I}\circ\mathcal{L}}$. 
The corresponding inequalities follow from Proposition \ref{corhip}. Having in mind that if $P=u\circ Q$, then $\check{P}=u\circ \check{Q}$, and if $A=t\circ B$, then $\widehat{A}=t\circ \widehat{B}$; the equality $\mathcal{I}\circ\mathcal{P}=\mathcal{P}_{\mathcal{I}\circ\mathcal{L}}$ follows immediately. The remaining norm inequalities follow from \cite[Proposition 3.7.(b)]{bpr2}.\end{proof}

The next natural step is to study composition ideals from the perspective of two-sided ideals. Sometimes composition ideals are polynomial two-sided ideals, for example the class ${\cal P}_{\cal K} = {\cal K} \circ {\cal P}$ of compact polynomials. But sometimes they are not two-sided ideals, for example the class ${\cal P}_{\cal W} = {\cal W} \circ {\cal P}$ of weakly compact polynomials (cf. Example \ref{expsi}). So, some extra condition on $\cal I$ is needed for ${\cal I} \circ {\cal P}$ to be a two-sided ideal. The rest of the paper is devoted to the search of such a condition and of operator ideals $\cal I$ satisfying it. The idea of the condition we shall provide, which is not only sufficient but also necessary, was based on \cite[Proposition 3.3]{sonia}.

By $\widehat{\otimes}^{n,s}_{\pi_s}E$ we denote the completed $n$-fold symmetric tensor product of $E$ endowed with the $s$-projective norm $\pi_s$ (see Floret \cite{klaus3}). The elementary symmetric tensor $x \otimes \stackrel{(n)}{\cdots}\otimes x$ shall be denoted by $\otimes^n x$. For $P \in {\cal P}(^nE;F)$, by $P_L$ we mean the linearization of $P$ on $\widehat{\otimes}^{n,s}_{\pi_s}E$, that is, $P_L\in \mathcal{L}(\widehat{\otimes}^{n,s}_{\pi_s}E;F)$, $P(x) = P_L(\otimes^n x)$ for every $x \in E$ and $\|P_L\| = \|P\|$. Given $u\in\mathcal{L}(E;F)$, there is a unique continuous linear operator $\otimes^{n,s}u\colon \widehat{\otimes}^{n,s}_{\pi_s}E\longrightarrow\widehat{\otimes}^{n,s}_{\pi_s}F$ such that $$\otimes^{n,s}u(\otimes^n x)=\otimes^n u(x)$$
for every $x \in E$ and $\|\otimes^{n,s}u\| = \|u\|^n$ \cite[Proposition 2.2(6)]{klaus3}.

\begin{definition}\rm A $p$-normed operator ideal $\mathcal{I}$ is said to be \textit{symmetrically tensorstable} ({\it s-tensorstable} in short) if there is a sequence $(C_n)_{n=1}^\infty$ of positive numbers such that, for all $n\in\mathbb{N}$ and $u\in\mathcal{I}(E;F)$, it holds $$\otimes^{n,s}u\in\mathcal{I}\left(\widehat{\otimes}^{n,s}_{\pi_s}E;\widehat{\otimes}^{n,s}_{\pi_s}F\right)\ \mbox{and}\ \|\otimes^{n,s}u\|_\mathcal{I}\le C_n\|u\|_\mathcal{I}^n.$$\end{definition}

Recall the notation ${\cal P}\circ{\cal I}$ introduced in Definition \ref{metfatpol}.

\begin{teo}\label{phipic}The following are equivalent for a $p$-normed ($p$-Banach) operator ideal $\mathcal{I}$:\\
{\rm (a)} $\mathcal{I}\circ\mathcal{P}$ is a $p$-normed ($p$-Banach) polynomial $(1,C_n)_{n=1}^\infty$-two-sided ideal.\\
{\rm (b)} $\mathcal{P}\circ\mathcal{I}\stackrel{(C_n)_{n}}{\longhookrightarrow}\mathcal{I}\circ\mathcal{P}$.\\
{\rm (c)} $\mathcal{I}$ is s-tensorstable with constants $(C_n)_{n=1}^\infty$.\end{teo}

\begin{proof}(a)$\Longrightarrow$(b) Given $P\in\mathcal{P}\circ\mathcal{I}(^nE;F)$,  there are a Banach space $G$, a linear operator $u\in\mathcal{I}(E;G)$ and a polynomial $Q\in\mathcal{P}(^nG;F)$ such that $P=Q\circ u$. Since the linear component of $\mathcal{I}\circ\mathcal{P}$ is the operator ideal $\mathcal{I}$ and $\mathcal{I}\circ\mathcal{P}$ is a polynomial two-sided ideal by assumption, we conclude that $P=Q\circ u$ belongs $\mathcal{I}\circ\mathcal{P}$ and $$\|P\|_{\mathcal{I}\circ\mathcal{P}}\le C_n\|Q\|\cdot\|u\|_{\mathcal{I}\circ\mathcal{P}}^n = C_n\|Q\|\cdot\|u\|_{\mathcal{I}}^n.$$ Taking the infimum over all factorizations it follows that $\|P\|_{\mathcal{I}\circ\mathcal{P}}\le C_n\|P\|_{\mathcal{P}\circ\mathcal{I}}$.\\
(b) $\Longrightarrow$ (a) According to the definition of two-sided ideals (or to the corresponding series criterion in the complete case) and using that $\mathcal{I}\circ\mathcal{P}$ is a $p$-normed ($p$-Banach) polynomial hyper-ideal (Theorem \ref{thipic}), we only need to check the right-hand side of  the two-sided ideal property. To do so, let $P\in \mathcal{I}\circ\mathcal{P}(^nE;F)$ and $R\in\mathcal{P}(^rF;H)$. There are a Banach space $G$, an $n$-homogenous polynomial $Q\in\mathcal{P}(^nE;G)$ and a linear operator $u\in\mathcal{I}(G;F)$ such that $P=u\circ Q$. By assumption, $R \circ u \in {\cal P} \circ {\cal I} \subseteq {\cal I} \circ {\cal P}$, so the hyper-ideal property of ${\cal I} \circ {\cal P}$ gives  $$R\circ P=(R\circ u)\circ Q \in {\cal I} \circ {\cal P}(^{rn}E;H),$$ and $$\|R\circ P\|_{\mathcal{I}\circ\mathcal{P}}\le\|R\circ u\|_{\mathcal{I}\circ\mathcal{P}}\cdot\|Q\|^r\le C_r\|R\circ u\|_{\mathcal{P}\circ\mathcal{I}}\cdot\|Q\|^r\le C_r\|R\|\cdot\|u\|_\mathcal{I}^r\cdot\|Q\|^r,$$ where the second inequality follows from the assumption. It follows that $\|R\circ P\|_{\mathcal{I}\circ\mathcal{P}}\le C_r\|R\|\cdot\|P\|_{\mathcal{I}\circ\mathcal{P}}^r$.\\
(b)$\Longrightarrow$(c) Let $u\in\mathcal{I}(E;F)$. Considering the canonical $n$-homogenous polynomial $\widehat{\sigma}_n^F\colon F\longrightarrow\widehat{\otimes}^{n,s}_{\pi_s}F$, given by $\widehat{\sigma}_n^F(x)=\otimes^n x$, by assumption we have
$$P:=\widehat{\sigma}_n^F\circ u\in \mathcal{P}\circ\mathcal{I}(^nE;\widehat{\otimes}^{n,s}_{\pi_s}F)\subseteq \mathcal{I}\circ\mathcal{P}(^nE;\widehat{\otimes}^{n,s}_{\pi_s}F)$$
and the corresponding norm inequality. By \cite[Propositions 3.2(b) and 3.7(b)]{bpr2} we have that $P_L\in \mathcal{I}(\widehat{\otimes}^{n,s}_{\pi_s}E;\widehat{\otimes}^{n,s}_{\pi_s}F)$ and $\|P_L\|_\mathcal{I}=\|P\|_{\mathcal{I}\circ\mathcal{P}}$. For every $x\in E$,  $$P_L(\otimes^n x)=\widehat{\sigma}_n^F\circ u(x)=\otimes^n u(x)=\otimes^{n,s}u(\otimes^n x).$$ As $P_L$ and $\otimes^{n,s}u$ are continuous linear operators, it follows that $$\otimes^{n,s}u = P_L\in\mathcal{I}(\widehat{\otimes}^{n,s}_{\pi_s}E;\widehat{\otimes}^{n,s}_{\pi_s}F).$$ Moreover, $$\|\otimes^{n,s}u\|_\mathcal{I}=\|P_{L}\|_\mathcal{I}=\|P\|_{\mathcal{I}\circ\mathcal{P}}\le C_n\|P\|_{\mathcal{P}\circ\mathcal{I}}\le C_n\|\widehat{\sigma}_n^F\|\cdot\|u\|_\mathcal{I}^n=C_n\|u\|_\mathcal{I}^n.$$
(c) $\Longrightarrow$ (b) Let $P\in\mathcal{P}\circ\mathcal{I}(^nE;F)$, $G$ be a Banach space, $u\in \mathcal{I}(E;G)$ and $Q\in\mathcal{P}(^nG;F)$ be such that $P=Q\circ u$. By assumption we know that $\otimes^{n,s}u\in\mathcal{I}(\widehat{\otimes}^{n,s}_{\pi_s}E;\widehat{\otimes}^{n,s}_{\pi_s}F)$ and $\|\otimes^{n,s}u\|_\mathcal{I}\le C_n\|u\|_\mathcal{I}^n$. Writing $$P=Q\circ u=(Q_L\circ\otimes^{n,s}u)\circ\widehat{\sigma}_n^E,$$ we get $P\in\mathcal{I}\circ\mathcal{P}(^nE;F)$ and $$\|P\|_{\mathcal{I}\circ\mathcal{P}}\le\|Q_L\|\cdot\|\otimes^{n,s}u\circ
\widehat{\sigma}_n^E\|_{\mathcal{I}\circ\mathcal{P}}\le\|Q\|\cdot \|\otimes^{n,s}u\|_{\mathcal{I}}\cdot\|\widehat{\sigma}_n^E\|\le C_n\|Q\|\cdot\|u\|_\mathcal{I}^n.$$
It follows that $\|P\|_{\mathcal{I}\circ\mathcal{P}}\le C_n\|P\|_{\mathcal{P}\circ\mathcal{I}}$.\end{proof}

Of course the proposition above can be used to give positive and negative examples. Let us start with a negative example regarding a celebrated operator ideal.

\begin{ex}\label{exabs}\rm Let $\Pi_p$, $p \geq 1$, be the Banach ideal of absolutely $p$-summing linear operators. By Theorem \ref{thipic}, $\Pi_p\circ\mathcal{P}$ is a Banach polynomial hyper-ideal. Consider the class $\mathcal{P}_{d,p}$ of $p$-dominated polynomials, which also goes back to Pietsch \cite{pietsch}. In \cite[Example 1]{botelho2} (see also \cite[Remark 47]{botelho1}), it is established that
$$P_n\colon\ell_1\longrightarrow\ell_1~,~P_n((\lambda_j)_{j=1}^\infty)=((\lambda_j)^n)_{j=1}^\infty,$$ is a $p$-dominated non-weakly compact $n$-polynomial. Since $\Pi_p \circ {\cal P}\subseteq \mathcal{W} \circ {\cal P}= \mathcal{P}_\mathcal{W}$, $P_n$ does not belong to $\Pi_p\circ\mathcal{P}$. From  $\mathcal{P}_{d,p}= \mathcal{P}\circ \Pi_p$ \cite[Proposition 46(a)]{botelho1}, it follows that $P_n$ belongs to $\mathcal{P}\circ \Pi_p$. Thus $\mathcal{P}\circ \Pi_p\nsubseteq\Pi_p\circ\mathcal{P}$, and from Theorem \ref{phipic} we conclude that $\Pi_p\circ\mathcal{P}$ is not a polynomial two-sided ideal. In Example \ref{extensorstable} we shall see that  $\Pi_p^{\rm dual}\circ\mathcal{P}$ is a Banach polynomial two-sided ideal.\end{ex}

Before giving positive examples, let us treat a case that, to the best of our knowledge, is open.

\begin{ex}\label{exex} \rm Let $\mathcal{CC}$ denote the closed ideal of completely continuous linear operators (weakly convergent sequences are sent to norm convergent sequences). Assuming that $\mathcal{CC}\circ\mathcal{P}$ is a polynomial two-sided ideal, by Theorem \ref{phipic}, $\mathcal{CC}$ would be s-tensorstable. Noting that $id_{\widehat{\otimes}^{n,s}_{\pi_s}E} = \otimes^{n,s}id_E$, in this case it would be true that the symmetric projective tensor product of a Schur space is a Schur space. But it is an open problem if the projective tensor product of a Schur space is a Schur space, so it is unknown if $\mathcal{CC}$ is s-tensorstable, that is, if $\mathcal{CC}\circ\mathcal{P}$ is a polynomial two-sided ideal. In our opinion, it is likely that the answer to the aforementioned open problem will turn out to be negative (see \cite[p.\,19]{ewerton}), so we conjecture that $\mathcal{CC}\circ\mathcal{P}$ is not a polynomial two-sided ideal. Recall the closed polynomial two-sided ideal ${\cal P}_{psc}$ of polynomially sequentially continuous polynomials (Proposition \ref{proppsc}).  It is easy to check that $\mathcal{CC}\circ\mathcal{P}\subseteq\mathcal{P}_{psc}$. The conjecture above yields that the inclusion is strict.\end{ex}

In order to give examples of operators ideal $\cal I$ for which ${\cal I} \circ {\cal P}$ is a polynomial two-sided ideal using Theorem \ref{phipic}, we recall the notion of {\it $\pi$-tensorstable ideals} \cite[Section 34]{klauslivro}, which we shall simply call {\it tensorstable ideals}. Denoting by $E_1\widehat{\otimes}_{\pi}\cdots\widehat{\otimes}_{\pi}E_n$ the completed $n$-fold projective tensor product of $E_1, \ldots, E_n$, given operators $u_1\in\mathcal{L}(E_1;F_1),\ldots, u_n\in\mathcal{L}(E_n;F_n)$, there is a unique operator $u_1\otimes\cdots\otimes u_n\colon E_1\widehat{\otimes}_{\pi}\cdots\widehat{\otimes}_{\pi}E_n$ $\longrightarrow F_1\widehat{\otimes}_{\pi}\cdots\widehat{\otimes}_{\pi}F_n$ such that $$u_1\otimes\cdots\otimes u_n(x_1\otimes\cdots\otimes x_n)=u_1(x_1)\otimes\cdots\otimes u_n(x_n),$$
for $x_1 \in E_1, \ldots, x_n \in E_n$. When $u = u_1 = \cdots = u_n$, we write simply $\otimes^n u$.

\begin{definition}\rm A $p$-normed operator ideal $\mathcal{I}$ is \textit{tensorstable} if there is a sequence $(C_n)_{n=1}^\infty$ of positive numbers such that, for $n\in\mathbb{N}$ and operators $u_1\in\mathcal{I}(E_1;F_1),\ldots u_n\in\mathcal{I}(E_n;F_n)$, it holds $$u_1\otimes\cdots\otimes u_n\in\mathcal{I}(E_1\widehat{\otimes}_{\pi}\cdots\widehat{\otimes}_{\pi}E_n;F_1\widehat{\otimes}_{\pi}\cdots\widehat{\otimes}_{\pi}F_n)$$ and
$\|u_1\otimes\cdots\otimes u_n\|_\mathcal{I}\le C_n\|u_1\|_\mathcal{I}\cdots\|u_n\|_\mathcal{I}$.  \end{definition}

The definition of tensorstability is usually presented for two operators, that is, for $n=2$ (see, e.g., \cite{defant2, klauslivro}). The associativity of the projective tensor product guarantees that this definition is equivalent to the one above.

\begin{prop}\label{ultpro} Every tensorstable operator ideal with constants $(C_n)_{n=1}^\infty$ is s-tensorsta-ble with constants $\ap\frac{C_n n^n}{n!}\fp_{n=1}^\infty$.\end{prop}
\begin{proof}Let $\cal I$ be a tensorstable operator ideal with constants $(C_n)_{n=1}^\infty$ and $u\in\mathcal{I}(E;F)$. Consider the inclusion operator $\iota_E^n\colon\widehat{\otimes}_{\pi_s}^{n,s}E\longrightarrow\widehat{\otimes}_\pi^nE$ and the symmetrization $n$-linear operator $$S_F^n\colon F^n\longrightarrow\widehat{\otimes}_{\pi_s}^{n,s}F~, ~S_F^n(y_1,\ldots,y_n)=\dfrac{1}{n!}\sum\limits_{\eta\in S_n}y_{\eta(1)}\otimes\cdots\otimes y_{\eta(n)}.$$ Call $(S_F^n)_L$ the linearization of $S_F^n$, that is, $(S_F^n)_L \in {\cal L}(\widehat{\otimes}_\pi^n F;\widehat{\otimes}_{\pi_s}^{n,s} F) $ and $(S_F^n)_L(x_1 \otimes \cdots \otimes x_n) = S_F^n(x_1, \ldots, x_n)$. For every $x \in E$, \begin{align*}(S_F^n)_L\circ\otimes^{n}u\circ\iota_E^n(\otimes^n x)&=(S_F^n)_L\circ\otimes^{n}u(\otimes^n x)=(S_F^n)_L(\otimes^n u(x))=\otimes^n u(x)=\otimes^{n,s}u(\otimes^n x).\end{align*} Since $(S_F^n)_L\circ\otimes^{n}u\circ\iota_E^n$ and $\otimes^{n,s}u$ are continuous linear operators, it follows that $\otimes^{n,s}u=(S_F^n)_L\circ\otimes^{n}u\circ\iota_E^n$. The tensorstability of $\cal I$ gives $\otimes^{n,s}u\in\mathcal{I}(\widehat{\otimes}^{n,s}_{\pi_s}E;\widehat{\otimes}^{n,s}_{\pi_s}F)$ and $$\|\otimes^{n,s}u\|_\mathcal{I}\le\|(S_F^n)_L\|\cdot\|\otimes^{n}u\|_\mathcal{I}\cdot\|\iota_E^n\|\le\dfrac{C_n n^n}{n!}\|u\|_\mathcal{I}^n,$$
because $\|\iota_E^n\|=1$ and $\|(S_F^n)_L\|=c(n,E)\le\dfrac{n^n}{n!}$ \cite[Proposition 2.3]{klaus3}.
\end{proof}

We finish the paper by providing examples of operator ideals $\cal I$ for which ${\cal I} \circ {\cal P}$ is a polynomial two-sided ideal.

\begin{ex}\label{extensorstable}\rm  The following Banach operator ideals $\cal I$ are tensorstable, hence s-tensor-stable by Proposition \ref{ultpro}, therefore ${\cal I} \circ {\cal P}$ is a Banach polynomial two-sided ideal by Theorem \ref{phipic}:\\
(a) The dual $\Pi_p^{\rm dual}$ of the ideal $\Pi_p$ of absolutely summing $p$-operators, which coincides with the maximal hull ${\cal K}_p^{\rm max}$ of the ideal ${\cal K}_p$ of $p$-compact operators \cite[Theorems 12,\,24,\,25]{pietschpams}.\\
(b) The closed ideal $\mathcal{S}$ of separable operators \cite[Example 3.5(a)]{sonia}.\\
(c) The ideals $\cal \overline{F}^{\|\cdot\|}$ of approximable operators and $\cal N$ of nuclear operators \cite[34.1]{klauslivro}.\\
(d) The ideal $\mathcal{J}$ of integral operators \cite[Theorem 2]{holub}.\\
(e) The ideal $\mathcal{L}_{\infty,q,\gamma}$ of Lorentz-Zygmund operators, $0<q\le1$ and $-1/q<\gamma<\infty$ \cite[Theorem 3.1]{resina}.\\
(f) The ideals $\mathcal{L}_{1,q}$, $q>1$, of $(1,q)$-factorable operators and $\mathcal{K}_{1,p}$, $p>1$, of $(1,p)$-compact operators \cite[Theorem 2.1]{defant2}.\\
(g) For $p \geq 1$, the dual $\mathcal{J}_p^{dual}$ of the ideal $\mathcal{J}_p$ of $p$-integral operators, and the surjective hull $\mathcal{K}_{1,p}^{sur}$ of the ideal $\mathcal{K}_{1,p}$ of $(1,p)$-compact operators \cite[Corollary 34.5.2(2)]{klauslivro}.

As Carl, Defant and Ramanujan perspicaciously observed in \cite{defant2}, the constants on the norm inequality appear naturally in the proof of the tensor stability.\end{ex}

\noindent{\bf Open question.} We do not know if every s-tensorstable operator ideal is tensorstable. If yes, Example \ref{exabs} will have a simpler reasoning (in \cite[Corollary 34.3.1]{klauslivro} it is proved that $\Pi_p$ is not tensorstable); and Theorem \ref{phipic} will have one more equivalent condition. But the reader should have in mind that things that work well for homogeneous polynomials might not work so well for multilinear operators. For example, Leung \cite{leung} proved that the dual $J'$ of the James space $J$ is symmetrically regular but fails to be regular.

\noindent Faculdade de Matem\'atica\\
Universidade Federal de Uberl\^andia\\
38.400-902 -- Uberl\^andia, Brazil\\
e-mails: botelho@ufu.br, ewertonrtorres@gmail.com.

\end{document}